\documentclass[11pt]{amsart}

\usepackage{latexsym}
\usepackage{amssymb}
\usepackage{amsmath}
\usepackage{color}

\newtheorem{theorem}{Theorem}[section]
\newtheorem{lemma}[theorem]{Lemma}
\newtheorem{proposition}[theorem]{Proposition}
\newtheorem{corollary}[theorem]{Corollary}

\newtheorem{remark}[theorem]{Remark}

\newtheorem{question}[theorem]{Question}

\newcommand\supp{\mathop{\rm supp}}

\newcommand\id{\mathop{\rm id}}

\newcommand\hh{\mathop{\rm h}}
\newcommand\eh{\mathop{\rm eh}}

\newcommand\nph{\varphi}

\newcommand\cb{\mathop{\rm cb}}

\newcommand\vn{\mathop{\rm VN}}

\newcommand\ball{\mathop{\rm ball}}

\newcommand{\cl}[1]{\mathcal{#1}}
\newcommand{\bb}[1]{\mathbb{#1}}

\newcommand\ah{\mathop{A_{\hh}}}
\newcommand\mah{\mathop{MA_{\hh}}}

\newcommand\ahg{\mathop{\ah(G)}}
\newcommand\mh{\mathop{\mah(G)}}
\newcommand\cbmh{\mathop{M^{\cb}A_{\hh}(G)}}

\newcommand\ev{\mathop{{\rm VN}_{\rm eh}}}

\newcommand{\ignore}[1]{ { } }
\newcommand{\M}{{\Bbb{M}}}
\newcommand{\cA}{{\mathcal{A}}}
\newcommand{\cB}{{\mathcal{B}}}

\newcommand{\bm}{{\frak m}}

\begin{document}

\title{Completely bounded bimodule maps and spectral synthesis}

\author[M. Alaghmandan]{M. Alaghmandan}
\address{Department of Mathematical Sciences,
Chalmers University of Technology and  the University of Gothenburg,
Gothenburg SE-412 96, Sweden}
\email{mahala@chalmers.se}

\author[I. G. Todorov]{I. G. Todorov}
\address{Pure Mathematics Research Centre, Queen's University Belfast, Belfast BT7 1NN, United Kingdom}
\email{i.todorov@qub.ac.uk}

\author[L. Turowska]{L. Turowska}
\address{Department of Mathematical Sciences,
Chalmers University of Technology and  the University of Gothenburg,
Gothenburg SE-412 96, Sweden}
\email{turowska@chalmers.se}

\date{31 December 2016}

\begin{abstract}
We initiate the study of the completely bounded multipliers of the Haagerup tensor product 
$A(G)\otimes_{\rm h} A(G)$
of two copies of the Fourier algebra $A(G)$ of a locally compact group $G$. 
If $E$ is a closed subset of $G$ we let $E^{\sharp} = \{(s,t) : st\in E\}$ and show that if $E^{\sharp}$ is a set of 
spectral synthesis for $A(G)\otimes_{\rm h} A(G)$ then $E$ is a set of local spectral synthesis for $A(G)$. 
Conversely, we prove that if $E$ is a set of spectral synthesis for $A(G)$ and $G$ is a Moore group then 
$E^{\sharp}$ is a set of spectral synthesis for $A(G)\otimes_{\rm h} A(G)$. 
Using the natural identification of the space of all completely bounded weak* continuous 
$\vn(G)'$-bimodule maps with the dual of $A(G)\otimes_{\rm h} A(G)$, we show that, 
in the case $G$ is weakly amenable, 
such a map leaves the multiplication algebra of $L^{\infty}(G)$ invariant if and only if its support 
is contained in the antidiagonal of $G$. 
\end{abstract}

\maketitle

\tableofcontents

\section{Introduction}\label{s_intro}

The connections between Harmonic Analysis and Operator Theory
originating from the seminal papers of W. Arveson \cite{a} and N. Varopoulos \cite{varopoulos}
have been fruitful and far-reaching. 
A particular instance of this interaction is the 
relation between Schur and Herz-Schur multipliers \cite{ch, haag}
that has been prominent in applications, for example to approximation properties 
of group operator algebras (see {\it e.g.} \cite{bo}). 
It is well-known that, given a locally compact second countable group $G$, the Schur multipliers on $G\times G$
can be identified with those (completely) bounded weak* continuous maps on the space $\cl B(L^2(G))$
of all bounded operators on $L^2(G)$ (here $G$ is equipped with left Haar measure)
that are also bimodular over $L^{\infty}(G)$, where the latter is viewed as an algebra of multiplication operators on $L^2(G)$. 
The right invariant part of the space of Schur multipliers  
(which arises from the functions $\nph$ on $G\times G$ that satisfy the condition $\nph(sr,tr) = \nph(s,t)$) 
consists precisely of those maps
that, in addition to the aforementioned properties, preserve the von Neumann algebra $\vn(G)$ of $G$.

The original motivation behind the present work was the development of 
a counterpart of the latter result in a setting where the places of $\vn(G)$ and $L^{\infty}(G)$ are exchanged. 
The space of all completely bounded weak* continuous $\vn(G)$-bimodule maps on $\cl B(L^2(G))$
has played a distinctive role in Operator Algebra Theory and have lately been prominent through the theory of 
locally compact quantum groups (see {\it e.g.} \cite{hnr} and \cite{jnr}). Those such 
maps that also preserve the multiplication 
algebra of $L^{\infty}(G)$ have been studied since the 1980's and are known to 
arise from regular Borel measures on $G$ (see \cite{gh, nrs, stor}).
However, a characterisation, analogous to the right invariance in the context of Schur multipliers -- 
and one that uses only harmonic-theoretic properties -- was not known. 
In the present paper, we establish such a characterisation and 
observe that it can be formulated in the language of spectral synthesis: it is
equivalent to the statement that the anitdiagonal of $G$ is a Helson set with respect to the
Haagerup tensor product $A(G)\otimes_{\rm h} A(G)$ of two copies of the Fourier algebra $A(G)$ of $G$. 
Our investigation highlights the connections between completely bounded bimodule maps
and spectral synthesis, which have not received substantial attention until now,
despite the importance of both notions in modern Analysis.

The aforementioned result 
required the development of a ground theory of bivariate Herz-Schur multipliers 
and served as a motivation to study questions of spectral synthesis in 
$A(G)\otimes_{\rm h} A(G)$.
Our results show that, with respect to spectral synthesis, the latter algebra is
better behaved than the seemingly more natural $A(G\times G)$, 
and point to substantial distinctions between these two algebras.
Indeed, for a vast class of groups we establish transference of spectral synthesis between $A(G)$ and $A(G)\otimes_{\rm h} A(G)$, 
while such result does not hold
for $A(G\times G)$ unless $G$ is virtually abelian.

In more detail, the paper is organised as follows. After collecting preliminaries and setting notation in Section \ref{s_prel}, 
we study, in Section \ref{s_aim}, the bivariate Fourier algebra $\ahg \stackrel{def}{=}A(G)\otimes_{\rm h} A(G)$ and establish some 
if its basic properties, highlighting the  
rather well-known fact that it is a regular commutative semi-simple Banach algebra with Gelfand spectrum $G\times G$.
Viewing $\ahg$ as a function algebra, we examine the space of its completely bounded multipliers, which can be thought of as 
a bivariate version of Herz-Schur multipliers, and show, among other things, that 
this algebra is weakly amenable if and only if the group $G$ is weakly amenable. 
We obtain a characterisation of the completely bounded multipliers of $\ahg$ in terms of 
(bounded) multipliers on products with finite groups, providing a version of a result from \cite{ch} (see Proposition \ref{p_McbAh(G)-finite-groups}). 
We show that the elements of the extended Haagerup tensor product $A(G)\otimes_{\eh}A(G)$ can be viewed as
separately continuous functions, an identification needed thereafter. 

In Section \ref{s_ss}, we study the question of spectral synthesis for $\ahg$. Note that the dual of $\ahg$
coincides with the extended Haagerup tensor product $\vn_{\rm eh}(G) \stackrel{def}{=} \vn(G)\otimes_{\eh}\vn(G)$
which, in turn, can be canonically identified, {\it via} a classical result of U. Haagerup's \cite{haag}, with the space of 
completely bounded weak* continuous $\vn(G)'$-bimodule maps (here $\vn(G)'$ denotes the commutant of $\vn(G)$). 
Thus, the classical theory of commutative Banach algebras allows us to associate to each such map its support,
a closed subset of $G\times G$. 
Viewing $\vn_{\rm eh}(G)$ as a (completely contractive) 
module over $\ahg$, we obtain bivariate versions of some classical results of P. Eymard \cite{eymard}. 
The main results in Section \ref{s_ss} are related to transference of spectral synthesis: 
associating to a subset $E\subseteq G$ the subset $E^{\sharp} = \{(s,t)\in G\times G : st\in E\}$ of $G\times G$, 
we show that if $E^{\sharp}$ is is a set of spectral synthesis for $\ahg$ then $E$ is a set of local spectral synthesis for $A(G)$. 
Conversely, if $E$ is a set of spectral synthesis for $A(G)$ and $G$ is a Moore group then $E^{\sharp}$ is a set of 
spectral synthesis for $\ahg$. Thus, for Moore groups, the sets $E$ and $E^{\sharp}$ satisfy spectral synthesis simultanenously. 
These results should be compared with other transference results in the literature, see {\it e.g.} \cite{lt}, \cite{stt} and \cite{st}, 
and are a part of a programme of relating harmonic analytic, one-variable, properties, to operator theoretic, two-variable, ones 
\cite{t_serdika}.

In Section \ref{s_vag}, we assume that $G$ is a virtually abelian group, and show that, in this case,
transference carries over to the set $E^* = \{(s,t)\in G\times G : ts^{-1}\in E\}$. 
This is obtained as a consequence of the fact that, for such groups, the flip of variables is a well-defined bounded map on 
$A_{\rm h}(G)$. 

Section \ref{s_ibm} is focused around the question of how the support of a map arising from an element 
of $\vn_{\eh}(G)$ influences the structure of the map. Our results demonstrate that the support contains information about the invariant 
subspaces of the map (see Theorem \ref{th_invars} and Corollaries \ref{c_no1} and \ref{c_no2}). 
As a consequence, we show that a completely bounded weak* continuous $\vn(G)'$-bimodule map 
leaves the multiplication algebra of $L^{\infty}(G)$ invariant if and only if its support is contained in the antidiagonal of $G$. 
This gives an intrinsic, harmonic analytic, characterisation of this class of maps. 

Operator space tensor products and, more generally, operator space theoretic concepts and results, 
play a prominent role in our approach. Our main references in this direction are \cite{blm} and \cite{er_book}.
In addition, we use in a crucial way results and techniques about masa-bimodules in $\cl B(L^2(G))$, whose basic theory 
was developed in \cite{a} and \cite{eks}.


\section{Preliminaries}\label{s_prel}

In this section, we introduce some basic concepts that will be needed in the sequel and set notation. 
For a normed space $\cl X$, we let $\ball(\cl X)$ be the unit ball of $\cl X$,
and $\cl B(\cl X)$ (resp. $\cl K(\cl X)$)
be the algebra of all bounded linear (resp. compact) operators on $\cl X$.
If $H$ is a Hilbert space and $\xi,\eta\in H$, we denote by
$\xi\otimes\eta^*$ the rank one operator on $H$ given by
$(\xi\otimes\eta^*)(\zeta) = (\zeta,\eta)\xi$, $\zeta\in H$.
By $\omega_{\xi,\eta}$ we denote the vector functional on $\cl B(H)$
defined by $\omega_{\xi,\eta}(T) = (T\xi,\eta)$.
The pairing between elements of a normed space
$\cl X$ and those of its dual $\cl X^*$ will be denoted by $\langle \cdot,\cdot\rangle_{\cl X,\cl X^*}$;
when no risk of confusion arises, we write simply $\langle \cdot,\cdot\rangle$.
By $\M_n(\cl X)$ we denote the
space of all $n$ by $n$ matrices with entries in $\cl X$; we set $\M_n = \M_n(\bb{C})$.
We let $CB(\cl X)$ be the (operator) space of all completely bounded maps on
an operator space $\cl X$.

The algebraic tensor product of vector spaces $\cl X$ and $\cl Y$
will be denoted by $\cl X\odot \cl Y$; if
$\cl X$ and $\cl Y$ are Banach spaces, we let
$\cl X\otimes_{\gamma}\cl Y$ be their Banach projective tensor product.
If $H$ and $K$ are Hilbert spaces, we denote by $H\otimes K$
their Hilbertian tensor product.
We let $\cl X\hat{\otimes} \cl Y$ denote the operator projective, and
$\cl X\otimes_{\rm h} \cl Y$ the Haagerup, tensor product of the operator spaces $\cl X$ and $\cl Y$.
By $\cl X\otimes_{\rm eh} \cl Y$ we will denote the extended
Haagerup tensor product of $\cl X$ and $\cl Y$; we refer the reader to \cite{er}
for its definition and properties.
If $\cl X$ and $\cl Y$ are dual operator spaces, their
weak* spacial tensor product will be denoted by $\cl X\bar\otimes\cl Y$,
and their $\sigma$-Haagerup tensor product by $\cl X\otimes_{\sigma\hh}\cl Y$.
Note that, in the latter case, $\cl X\otimes_{\rm eh} \cl Y$ coincides with the
weak* Haagerup tensor product of $\cl X$ and $\cl Y$ introduced in \cite{bs}.
We often use the same symbol to denote both a bilinear map and its linearisation
through a tensor product.

Recall that a Banach algebra $\cA$ equipped with an operator space structure is called
\emph{completely contractive}  if
 \[
 \| [a_{i,j} b_{k,l} ] \|_{\M_{mn}(\cA)} \leq \| [a_{i,j}]\|_{\M_n(\cl A)} \|[b_{k,l}]\|_{\M_m(\cA)}
 \]
for every $ [a_{i,j}] \in \M_n(\cA)$ and $[b_{k,l}] \in \M_m(\cA)$ and $n,m \in \Bbb{N}$.
Thus, if $\cl A$ is a completely contractive  Banach algebra then the linearisation of the product
extends to a completely contractive map $m_{\cl A} : \cl A\hat{\otimes}\cl A\to \cl A$.

Let $\cl A$ be a commutative regular semi-simple 
completely contractive Banach algebra with Gelfand spectrum $\Omega$;
thus, $\cl A$ can be thought of as a subalgebra of the algebra $C_0(\Omega)$
of all continuous functions on $\Omega$ vanishing at infinity.
A continuous function $b : \Omega\to\bb{C}$ is called a \emph{multiplier} of
$\cl A$ if $b\cl A\subseteq \cl A$; in this case, we have a well-defined 
map $\bm_b$ on $\cl A$, given by $\bm_b(a) = ba$, which is
automatically bounded.
If the map $\bm_b$ is moreover completely bounded, $b$ is called
a \emph{completely bounded multiplier}.
We denote by  $M\cA$ (resp. $M^{\cb}\cA$) the space of all
multipliers (resp. completely bounded multipliers) of $\cl A$.
It is known that a (bounded) linear map $T : \cl A\to \cl A$ is of the form $T  = \bm_b$ for some $b\in M\cl A$
if and only if $T(x)y = xT(y)$ for all $x,y\in \cA$ (see {\it e.g.} \cite[Proposition~2.2.16]{kaniuth}).
Note that $M\cA$ (resp. $M^{\cb}\cA$) is a closed subalgebra of $\mathcal{B}(\cA)$
(resp. $CB(\cA)$).
If $b\in M^{\cb}\cA$, we denote by
$\|b\|_{\rm cbm}$ the completely bounded norm of $\bm_b$; we often identify the functions $b\in M^{\cb}\cl A$
with the corresponding linear transformations $\bm_b$.
Note that, if $a\in \cl A$, then
\[
\| \bm_a^{(n)} [a_{k,l} ] \|_{\M_{n}(\cA)} = \|  [a a_{k,l} ] \|_{\M_{n}(\cA)}  \leq \| a \|_{\cA} \|[a_{k,l}]\|_{\M_n(\cA)}
\]
for every $[a_{k,l}] \in \M_n(\cA)$ and every $n \in \Bbb{N}$.
Therefore,  the mapping $a \mapsto \bm_a$ from $\cl A$ into $M^{\cb} \cA$ is a contraction.

We next recall some basic facts from \cite{eymard} and \cite{ch}.
Let $G$ be a locally compact group. The Haar measure evaluated at a Borel set $E\subseteq G$
will be denoted by $|E|$, and integration with respect to it along the variable $s$
will be denoted by $ds$. As customary, $a\ast b$ denotes the convolution, whenever defined,
of the functions $a$ and $b$.
For $t\in G$, we let $\lambda_t$ be the unitary operator
on $L^2(G)$, given by $\lambda_t f(s) = f(t^{-1}s)$, $s\in G$, $f\in L^2(G)$.
We let $M(G)$ be the Banach *-algebra of all complex Borel measures on $G$
and use the symbol $\lambda$ to denote
the left regular *-representation of $M(G)$ on $L^2(G)$;
thus,
$$(\lambda(\mu) f)  = \int_G \lambda_sf  d\mu(s), \ \ \ \mu\in M(G), f\in L^2(G),$$
where the integral is understood in the weak sense.
We identify $L^1(G)$ with a *-subalgebra of $M(G)$ in the canonical way.

We let $\vn(G)$ (resp. $C^*_r(G)$, $C^*(G)$) be
the von Neumann algebra (resp. the reduced C*-algebra, the full C*-algebra) of $G$. As usual,
$A(G)$ (resp. $B(G)$) stands for the Fourier (resp. the Fourier-Stieltjes) algebra of $G$.
Thus,
$$C_r^*(G) = \overline{\{\lambda(f) : f\in L^1(G)\}}, \ \ \vn(G) = \overline{C_r^*(G)}^{w^*},$$
$$B(G) = \{(\pi(\cdot)\xi,\eta) : \pi : G\to \cl B(H) \mbox{ cont. unitary representation}, \xi,\eta\in H\},$$
and
$A(G)$ is the collection of the functions on $G$ of the form $s\to (\lambda_s\xi,\eta)$, where $\xi,\eta\in L^2(G)$;
see \cite{eymard} for details.
We denote by $\|\cdot\|_{A}$ the norm of $A(G)$.
Note that the dual of $A(G)$ (resp. $C^*(G)$) can be canonically identified with $\vn(G)$ (resp. $B(G)$).
More specifically, if $\phi\in A(G)$ and $\xi,\eta\in L^2(G)$ are such that
$\phi(s) = (\lambda_s\xi,\eta)$, $s\in G$, then
\begin{equation}\label{eq_vndu}
\langle \phi,T\rangle = (T\xi,\eta), \ \ \ T\in \vn(G).
\end{equation}
We equip $A(G)$ (resp. $B(G)$) with the operator space structure
arising from the latter identification.
Note that both $A(G)$ and $B(G)$ are completely contractive Banach algebras
with respect to these operator space structures.
For each  $\psi\in MA(G)$, the dual $\frak{m}_\psi^*$ of the map $\frak{m}_\psi$ acts on $\vn(G)$;
in fact, $\frak{m}_\psi^*(\lambda_t) = \psi(t)\lambda_t$, $t\in G$, and $\frak{m}_\psi^*(\lambda(f)) = \lambda(\psi f)$,
$f\in L^1(G)$.
Note that a multiplier $\psi\in MA(G)$ is completely bounded
precisely when $\frak{m}_\psi^*$ is completely bounded; in this case, $\|\psi\|_{\rm cbm} = \|\frak{m}_\psi^*\|_{\rm cb}$.
We set $\psi \cdot T = \frak{m}_\psi^*(T)$.
The elements of $M^{\cb}A(G)$ are called Herz-Schur multipliers and were introduced and originally studied in \cite{ch}.

Let $H$ and $K$ be separable Hilbert spaces and
$\cl M\subseteq \cl B(H)$ and $\cl N\subseteq \cl B(K)$ be von Neumann algebras.
Every element $u\in \cl M\otimes_{\eh}\cl N$
has a representation
\begin{equation}\label{eq_cs}
u = \sum_{i =1}^{\infty} a_i\otimes b_i,
\end{equation}
where $(a_i)_{i\in\bb{N}} \subseteq \cl M$ and $(b_i)_{i \in \bb{N}} \subseteq \cl N$
are sequences such that
$\sum_{i=1}^{\infty} a_ia_i^*$ and $\sum_{i = 1}^{\infty} b_i^*b_i$
are weak* convergent.
In this case,  the series (\ref{eq_cs}) converges in the weak* topology of
$\cl M\otimes_{\eh}\cl N$ with respect to the
completely isometric identification \cite{bs}
\begin{equation}\label{eq_mnehd}
\cl M\otimes_{\eh}\cl N \equiv (\cl M_*\otimes_{\hh}\cl N_*)^*,
\end{equation}
where $\cl M_*$ and $\cl N_*$ denote the preduals of $\cl M$ and $\cl N$, respectively.
Following \cite{bs}, call (\ref{eq_cs}) a w*-representation of $u$.
Denoting by $A$ (resp. $B$) the row (resp. column) operator $(a_i)_{i\in \bb{N}}$
(resp. $(b_i)_{i \in \bb{N}}$), we write (\ref{eq_cs}) as $u = A\odot B$.
Every such $u$ gives rise to a completely bounded
weak* continuous $\cl M',\cl N'$-module map
$\Phi_u : \cl B(K,H)\to \cl B(K,H)$
given by
\begin{equation}\label{eq_Phi_u}
\Phi_u(T) = \sum_{i =1}^{\infty} a_i T b_i, \ \ \ T\in \cl B(K,H),
\end{equation}
and the map $u\to \Phi_u$ is a complete isometry from $\cl M\otimes_{\eh}\cl N$ onto
the space $CB_{\cl M',\cl N'}^{w^*}(\cl B(K,H))$ of all weak* continuous completely bounded
$\cl M',\cl N'$-module maps on $\cl B(K,H)$ \cite{bs}.
Note that the algebraic tensor product
$\cl M\odot\cl N$ can be viewed in a natural way as a (weak* dense) subspace
of $\cl M\otimes_{\eh}\cl N$.


\section{Multipliers of bivariate Fourier algebras}\label{s_aim}

In this section, we introduce a natural bivariate version of Herz-Schur multipliers
and develop their basic properties. 
We set
$$\ahg = A(G)\otimes_{\hh} A(G) \ \mbox{ and } \ \ev(G) = \vn(G)\otimes_{\eh}\vn(G).$$
According to (\ref{eq_mnehd}), we have a completely isometric identification
\begin{equation}\label{eq_dah}
\ahg\mbox{}^* \ \equiv \ev(G);
\end{equation}
under this identification,
\begin{equation}\label{eq_lambdas}
\langle \phi \otimes \psi,\lambda_s\otimes\lambda_t\rangle = \phi(s)\psi(t),
\  \ \ \phi,\psi\in A(G), s,t\in G.
\end{equation}

We proceed with some certainly well-known considerations;
because of the frequent lack of precise references, we provide the full details,
which also serve our aim to set the appropriate context and notation for their
subsequent applications.
We first note that the natural injection
$$\iota_G : A(G)\to C_0(G)$$
is completely contractive. Indeed, let $U = [u_{i,j}]_{i,j}\in \M_n(A(G))$
and associate to $U$ the map $F_U : \vn(G)\to \M_n$ given by
$F_U(T) = [\langle u_{i,j},T\rangle]_{i,j}$.
Then
\begin{eqnarray*}
\|\iota_G^{(n)}(U)\|_{\M_n(C_0(G))}
& =& \sup_{s\in G} \|[u_{i,j}(s)] \|_{\M_n}
= \sup_{s\in G} \|[\langle u_{i,j},\lambda_s\rangle] \|_{\M_n}\\
& \leq & \sup\{\|[\langle u_{i,j},T\rangle] \|_{\M_n} : T\in \ball(\vn((G))\}\\
& = & \sup\{\|F_U(T)\|_{\M_n} : T\in \ball(\vn((G))\}\\
& = & \|F_U\| \leq \|F_U\|_{\cb} = \|U\|_{\M_n(A(G))}.
\end{eqnarray*}
Thus, the map
\begin{equation}\label{eq_hgh0}
\iota_{\hh} \stackrel{def}{=} \iota_G\otimes_{\hh}\iota_G : \ahg \to C_0(G)\otimes_{\hh} C_0(G)
\end{equation}
is completely contractive.
On the other hand, there is a natural contractive injection 
\begin{equation}\label{eq_var}
C_0(G)\otimes_{\hh} C_0(G) \to C_0(G) \otimes_{\min} C_0(G) \equiv C_0(G\times G),
\end{equation}
which allows us to view the elements of $C_0(G)\otimes_{\hh} C_0(G)$ as
continuous functions on $G\times G$ (vanishing at infinity).
In fact, $C_0(G)\otimes_{\hh} C_0(G)$ is a (Banach) algebra under pointwise addition and
multiplication and, by the Grothendieck inequality, coincides up to renorming
with the Varopoulos algebra $C_0(G)\otimes_{\gamma} C_0(G)$.
If $v\in \ahg$ then, in view of (\ref{eq_hgh0}) and (\ref{eq_var}),
\begin{equation}\label{eq_infh}
\|\iota_{\hh}(v)\|_{\infty}\leq \|v\|_{\hh}.
\end{equation}

The duality in the next lemma is the one arising from the
identification (\ref{eq_dah}).

\begin{lemma}\label{l_iotah}
If $s,t\in G$ and $v\in \ahg$ then
\begin{equation}\label{eq_wst}
\langle v, \lambda_s\otimes\lambda_t\rangle = \iota_{\hh}(v)(s,t).
\end{equation}
In particular, the map $\iota_{\hh}$ is injective.
\end{lemma}
\begin{proof}
Let $s,t\in G$, and suppose that $v = \phi \otimes \psi$ for some $\phi,\psi\in A(G)$.
Then
$$\langle v, \lambda_s\otimes\lambda_t\rangle
= \langle \phi \otimes \psi, \lambda_s\otimes\lambda_t\rangle
= \langle \phi, \lambda_s\rangle \langle \psi,\lambda_t\rangle
= \phi(s) \psi(t) = \iota_{\hh}(\phi\otimes \psi)(s,t).$$
It follows that (\ref{eq_wst}) holds for all $v\in A(G)\odot A(G)$.
For every $T\in \ev(G)$, the map
$v \to \langle v, T\rangle$ is norm  continuous.
On the other hand, in view of (\ref{eq_var}),
$$|\iota_{\hh}(v)(s,t)| \leq \|\iota_{\hh}(v)\|_{\infty} \leq \|\iota_{\hh}(v)\|_{\hh}.$$
Identity (\ref{eq_wst}) now follows from the density of $A(G)\odot A(G)$ in
$\ahg$.

Assuming $\iota_{\hh}(v) = 0$, we have that $\iota_{\hh}(v)(s,t) = 0$ for all $s,t\in G$.
By (\ref{eq_wst}), $\langle v, \lambda_s\otimes\lambda_t\rangle = 0$ for all $s,t\in G$.
An application of Kaplansky's Density Theorem shows that
the set $\{\lambda_s\otimes \lambda_t : s,t\in G\}$
spans a weak* dense subspace of $\ev(G)$; it now follows that $v = 0$.
\end{proof}

Since the Haagerup norm is dominated by the operator projective one,
the identity map on $A(G)\odot A(G)$ extends to a complete contraction
\begin{equation}\label{eq_projective_to_Ah}
\hat{\iota} : A(G)\hat{\otimes} A(G) \to \ahg.
\end{equation}
Identifying $A(G)\hat{\otimes} A(G)$ with $A(G\times G)$
(see \cite[Chapter~16]{er_book}), we thus consider $\hat{\iota}$ as a complete contraction
from $A(G\times G)$ into $\ahg$.
Note that
$\iota_{\hh}\circ \hat{\iota} = \iota_{G\times G}$;
indeed, the latter identity is straightforward on the algebraic tensor product
$A(G)\odot A(G)$, and hence holds by density and continuity.
Since $\iota_{G\times G}$ is injective, we conclude that $\hat{\iota}$ is
injective.
The dual $\hat{\iota}^* : \ev(G) \to \vn(G\times G)$ of the map
$\hat{\iota} : A(G\times G)\to \ahg$ is easily seen to coincide with the canonical inclusion of
$\ev(G)$ into $\vn(G\times G)$, and is hence 
(completely contractive and) injective \cite[Corollary 3.8]{bs}.

In the sequel, we often suppress the notations $\iota_{G\times G}$,
$\hat{\iota}$ and $\iota_{\hh}$ and, by virtue of Lemma \ref{l_iotah},
consider the elements of $\ahg$ as (continuous) functions
on $G\times G$.

The next proposition contains the main facts that we will need about the algebra $\ahg$.
Recall that a normed  algebra $(\mathcal{A}, \|\cdot\|_{\mathcal{A}})$ is said to have a
\emph{(left) bounded approximate unit} \cite{w}, if there exists a constant $C>0$
so that for every $v \in \mathcal{A}$ and every $\epsilon > 0$, there exists $u \in \mathcal{A}$
such that $\|u\|_{\mathcal{A}}\leq C$ and $\| uv- v\|_{\mathcal{A}}<\epsilon$.

\begin{proposition}\label{p_ccehag} 
The following statements hold true:

(i) \ \  The space $\ahg$ is a regular semisimple Tauberian
completely contractive Banach algebra with respect to the
operation of pointwise multiplication,
whose Gelfand spectrum can be identified with $G \times G$.

(ii) \ The map $\iota_{\hh}$ is an algebra homomorphism.

(iii) The algebra $\ahg$ has a bounded approximate identity if and only if $G$ is amenable. 
Furthermore, if $G$ is amenable then the bounded approximate identity can be chosen to be compactly supported.
\end{proposition}
\begin{proof}
(i), (ii) Let
$$m : (A(G)\odot A(G))\times (A(G)\odot A(G)) \to A(G)\odot A(G)$$
be the map given by
$$m(\phi \otimes \psi, \phi'\otimes \psi') = (\phi\phi')\otimes (\psi\psi').$$
By \cite[Section 9.2]{daws}, $m$ linearises to a
completely bounded bilinear map
$$m_{\hh} : \ahg\hat{\otimes} \ahg  \to \ahg,$$
turning $\ahg$ into a commutative completely contractive Banach algebra.

Let $w = \phi \otimes \psi$ and $w' = \phi'\otimes \psi'$ for some $\phi,\psi,\phi',\psi'\in A(G)$; then
\begin{eqnarray*}
\iota_{\hh}(m(w,w'))(s,t)
& = & \iota_{\hh}((\phi\phi')\otimes (\psi\psi'))(s,t)
= (\phi\phi')(s)(\psi\psi')(t)\\
& = & (\phi\otimes \psi)(s,t) (\phi'\otimes \psi')(s,t)
=
\iota_{\hh}(w)(s,t) \iota_{\hh}(w')(s,t).
\end{eqnarray*}
By the continuity of $m$, $\iota_{\hh}(m(w,w')) = \iota_{\hh}(w)\iota_{\hh}(w')$ for all $w,w'\in \ahg$ and
$\iota_{\hh}$ is a homomorphism. 
Therefore $m$ coincides with the pointwise multiplication.
The fact that the Gelfand spectrum of $\ahg$ coincides with $G\times G$ follows from \cite[Theorem~2]{tomiyama}.
Since $\iota_{G\times G}$ is injective and $A(G\times G)$ is a regular Banach algebra,  we conclude that $\ahg$ is regular, too. Note that, since the elements $\lambda_s\otimes \lambda_t$, $s,t\in G$, are
characters of $\ahg$, the latter algebra is also semi-simple.

Note that the space $\cl X = A(G) \cap C_c(G)$ is dense in $A(G)$; it follows that
the space $\cl X \odot \cl X$ is dense in $\ahg$. This implies that
$C_c(G \times G) \cap \ahg$ is dense in $\ahg$, that is, $\ahg$ is Tauberian.

(iii)
Suppose that $G$ is  amenable. By Leptin's Theorem,
$A(G)$ has a bounded approximate identity say $(\phi_\alpha)_\alpha$.
Set $w_{\alpha} = \phi_{\alpha}\otimes \phi_{\alpha}$.
If $\psi_1,\psi_2\in A(G)$ then, clearly,
$w_{\alpha}(\psi_1\otimes \psi_2)\to_{\alpha} \psi_1\otimes \psi_2$ in $\ahg$.
Now a straightforward approximation argument
shows that $(w_{\alpha})_{\alpha}$ is a (bounded) approximate identity for $\ahg$.

Conversely, suppose that $(w_\alpha)_{\alpha}$ is a bounded approximate identity of $\ahg$.
Let $\delta_s$ denote the character on $A(G)$ corresponding to an element $s\in G$.
The map $\id \otimes \delta_s : \ahg\to A(G)$ is a (completely) contractive homomorphism.
For an arbitrary  $0 \neq v \in A(G)$, let $s\in G$ so that $v(s) \neq 0$.
Note that
\begin{eqnarray*}
& & (\id \otimes \delta_s)(w_\alpha) v = (\id\otimes \delta_s)(w_\alpha) (\id\otimes \delta_s)(v \otimes v(s)^{-1} v)\\
& = & (\id \otimes \delta_s) (w_\alpha (v\otimes v(s)^{-1} v))
\rightarrow\mbox{}_{\alpha} (\id \otimes \delta_s) (v \otimes v(s)^{-1} v) = v.
\end{eqnarray*}
Thus, $A(G)$ has a (left) bounded approximate unit. By \cite[Theorem~1]{w},
$A(G)$ has a bounded approximate identity.
By Leptin's Theorem, $G$ is amenable.
\end{proof}

The following lemma will be needed shortly,
but it may be interesting in its own right.

\begin{lemma}\label{l:homomorphism-extension}
Let $\cA$  be a commutative Banach algebra,
$\cB$ be a completely contractive commutative Banach algebra, and
$\theta :\cA \rightarrow M^{\cb} \cB$ be a bounded homomorphism.   If $\cA$ has a bounded approximate identity and the linear span of
$\{ \theta(a)b : \ a\in \cA,\ b\in \cB\}$ is dense in $\cB$,  then $\theta$ can be extended to a  bounded map $\theta : M \cA \rightarrow M^{\cb} \cB$.
In particular, if $\cA$ is a completely contractive Banach algebra with  a bounded approximate identity, then $M\cA = M^{\cb} \cA$.
\end{lemma}
\begin{proof}
Fix a bounded approximate identity $(a_\alpha)_\alpha$  of $\cA$.
Let
$$\cB_0 = {\rm span}\{ \theta(a)(b): \ a\in \cA,\ b\in \cB\}.$$
For a given $c \in M \cA$, define $\theta(c)$ on $\cB_0$ by
 \[ \theta(c) \left(\sum_{k=1}^m \theta(a_k)(b_k)\right) := \sum_{k=1}^m \theta(ca_k)( b_k), \ \ \ a_k\in \cl A, b_k\in \cl B, k = 1,\dots,m.
 \]
The mapping $\theta(c)$ is  a well-defined linear map on $\cB_0$.
In fact, if
\[
\sum_{k=1}^n\theta(a^{(1)}_k)b^{(1)}_k = \sum_{l=1}^m\theta(a^{(2)}_k)b^{(2)}_k,
\]
 for some subsets
 $\{a^{(1)}_k, a^{(2)}_l: k=1, \ldots,n, \; l=1, \ldots, m\} \subseteq \cA$  and
 $\{b^{(1)}_k,b^{(2)}_l: k=1, \ldots,n, \; l=1, \ldots, m\} \subseteq \cB$,
then
\begin{eqnarray}\label{eq:well-definness-theta}
\sum_{k=1}^n \theta(ca_k^{(1)})b^{(1)}_k  &=& \lim_\alpha \sum_{k=1}^n \theta(ca_\alpha  a_k^{(1)})b^{(1)}_k \nonumber \\
&=& \lim_\alpha  \theta(ca_\alpha)\left( \sum_{k=1}^n \theta( a_k^{(1)})b^{(1)}_k\right) \\
&=& \lim_\alpha \theta(ca_\alpha)  \left(\sum_{l=1}^m \theta( a_l^{(2)})b^{(2)}_l \right) \nonumber\\
&=& \lim_\alpha   \sum_{l=1}^m \theta(ca_\alpha a_l^{(2)})b^{(2)}_l   =  \sum_{l=1}^m \theta(ca_l^{(2)})b^{(2)}_l.\nonumber
\end{eqnarray}

We claim that $\theta(c)$ is a completely bounded map on $\cB_0$.
Let
\[
\left[ \sum_{k=1}^{n_{i,j} } \theta(a_k^{(i,j)})b^{(i,j)}_k\right]_{i,j}
\]
 be an arbitrary element in the unit ball of $\Bbb{M}_n(\cB_0)$.  Then
\begin{eqnarray*}
& & \left\| \theta(c)^{(n)} \left[ \sum_{k=1}^{n_{i,j} } \theta(a_k^{(i,j)})b^{(i,j)}_k\right]_{i,j}\right\| =
\left\|  \left[ \sum_{k=1}^{n_{i,j} } \theta(c a_k^{(i,j)})b^{(i,j)}_k\right]_{i,j}\right\|\\
& = & \lim_\alpha \left\|  \left[ \sum_{k=1}^{n_{i,j} } \theta(c a_\alpha a_k^{(i,j)})b^{(i,j)}_k\right]_{i,j}\right\|\\
& = & \lim_\alpha \left\|  \theta(c a_\alpha)^{(n)} \left[ \sum_{k=1}^{n_{i,j} } \theta(a_k^{(i,j)})b^{(i,j)}_k\right]_{i,j}\right\| \\
&\leq &   \sup_\alpha  \| \theta(ca_\alpha)\|_{\rm cbm}
\leq  \| \theta\| \|c\|_{M \cA} \sup_\alpha \|a_\alpha\| < \infty.
\end{eqnarray*}
Since $\cl B_0$ is dense in $\cl B$, the map
$\theta(c)$ can be extended as a completely bounded map (denoted in the same way) on $\cB$.
Furthermore, $\theta(c)$ is a multiplier.
In fact, let $b, b' \in \cB$.
Since $\cB_0$ is dense in $\cB$, there is a sequence
$\left( \sum_{k=1}^{n_i} \theta(a_k^{(i)})b_k^{(i)}\right)_{i \in \Bbb{N}}$ in $\cB_0$
converging to $b$.
We have
\begin{eqnarray*}
\theta(c)(bb') &=& \lim_{i\to\infty}  \theta(c)\left( \sum_{k=1}^{n_i} \theta(a_k^{(i)}) b_k^{(i)} b'\right)\\
&=& \lim_{i\to\infty}  \sum_{k=1}^{n_i} \theta(ca_k^{(i)}) b_k^{(i)} b' \\
&=& \lim_{i\to\infty} \left( \sum_{k=1}^{n_i} \theta(ca_k^{(i)}) b_k^{(i)}\right) b'\\
&=& \lim_{i\to\infty} \theta(c)\left( \sum_{k=1}^{n_i} \theta(a_k^{(i)}) b_k^{(i)}\right) b' = \theta(c)(b) b';
\end{eqnarray*}
thus, $\theta$ takes values in $M^{\cb}\cB$.

To prove the last statement in the formulation of the Lemma,
note that if $\cA$ is a completely contractive Banach algebra, $\cA$ sits inside $M^{\cb} \cA$
in a natural fashion.
Since $\cl A$ possesses a  (bounded) approximate identity, the set
$\{ab: a,b \in \cA\}$ is dense in $\cA$.
By the first part of the proof,
the identity map can be extended to a map $\theta: M\cA \rightarrow M^{\cb} \cA$ where for each  $b \in M \cA$,
 \[
 \theta(b)(a) = \lim_\alpha (ba_\alpha)a = ba.
 \]
Therefore, the extension $\theta$ is still the identity map, and hence $M\cA \subseteq M^{\cb} \cA$.
This completes the proof as the inclusion $M^{\cb} \cA \subseteq M \cA$ holds by definition.
\end{proof}

 \begin{remark}
{\rm Lemma~\ref{l:homomorphism-extension} was formulated in the generality that is needed later, that is,
for the case $\cA$ is commutative and $\cB$ is  commutative and completely contractive.
However, it holds more generally when
 $\cA$ is an arbitrary Banach algebra and $\cB$ is a completely bounded Banach algebra.
We refer the reader to \cite{daws} for more details on completely bounded multipliers of completely bounded Banach algebras.}
 \end{remark}

We note that
\begin{equation}\label{eq_ahginm}
\ahg \subseteq M^{\cb}\ahg \subseteq M\ahg,
\end{equation}
where the first inclusion follows from the fact that $\ahg$ is a
completely contractive Banach algebra (see Proposition \ref{p_ccehag} (i)).
The following corollary is immediate from Lemma~\ref{l:homomorphism-extension} and Proposition~\ref{p_ccehag}.

\begin{corollary}\label{c:AM-McbAh=MAh}
If $G$ is an amenable locally compact group then $M^{\cb}\ahg$ $=$ $M\ahg$.
\end{corollary}

\begin{proposition}\label{p_algtenin}
The following hold:

(i) \  $M^{\rm cb}A(G) \odot M^{\rm cb}A(G) \subseteq M^{\rm cb}A_{\hh}(G)$;

(ii) $\ahg \subseteq \overline{M^{\rm cb}A(G) \odot M^{\rm cb}A(G)}^{\|\cdot\|_{\rm cbm}}$.

\noindent Moreover, if $f,g\in M^{\rm cb}A(G)$, then
$\|f \otimes g\|_{\rm cbm}\leq \|f\|_{\rm cbm}\|g\|_{\rm cbm}.$
\end{proposition}
\begin{proof}
Let $f,g\in M^{\rm cb}A(G)$. Then the map $\bm_f : A(G)\to A(G)$, given by $\bm_f(h) = fh$,
is completely bounded. Thus, $\bm_f\otimes \id : \ahg\to \ahg$ is completely bounded;
however, it is easy to note that $(\bm_f\otimes \id)(v) = (f\otimes 1)v$, $v\in \ahg$.
Thus, $f\otimes 1\in M^{\cb}\ahg$. By symmetry, $1\otimes g\in M^{\cb}\ahg$ and hence
$f\otimes g = (f\otimes 1)(1\otimes g)\in M^{\cb}\ahg$.
The norm inequality is straightforward from the fact that $M^{\cb}A_{\hh}(G)$ is a Banach algebra.

Since $\ahg$ is a completely contractive Banach algebra, if $v\in \ahg$ then $\|v\|_{\rm cbm} \leq \|v\|_{\hh}$.
Now the fact that $A(G)\subseteq M^{\cb}A(G)$ implies
$$\ahg = \overline{A(G)\odot A(G)}^{\|\cdot\|_{\hh}}\subseteq \overline{M^{\rm cb}A(G) \odot M^{\rm cb}A(G)}^{\|\cdot\|_{\rm cbm}}.$$
\end{proof}

Since $\|\phi\|_{\infty}\leq \|\phi\|_{B(G)}$ whenever $\phi \in B(G)$ (see \cite[Corollary~1.8]{ch}),
a straightforward argument shows that, if $w = \sum_{i=1}^{\infty}\phi_i\otimes\psi_i$ is an element 
of $B(G)\otimes_{\gamma} B(G)$ (where we have assumed that $\sum_{i=1}^{\infty} \|\phi_i\|^2_{B(G)} < \infty$
and $\sum_{i=1}^{\infty} \|\psi_i\|^2_{B(G)} < \infty$)
then the series $\sum_{i=1}^{\infty} \phi_i(s)\psi_i(t)$ converges for all $s,t\in G$;
thus, $w$ can be identified with a function on $G\times G$. 
Proposition \ref{p_algtenin} now implies that 
\[
B(G) \otimes_\gamma B(G) \subseteq \overline{M^{\rm cb}A(G) \odot M^{\rm cb}A(G)}^{\|\cdot\|_{\rm cbm}} \subseteq M^{\rm cb}\ahg.
\]
It is natural to ask whether $M^{\cb}\ahg$ can be obtained from the two copies of $M^{\cb}A(G)$
lying inside it. More specifically, we formulate the following question.

\begin{question}\label{q_bgh}
(i) \ Is it true that $B(G)\otimes_{\hh} B(G)\subseteq \cbmh$?

(ii) Is $M^{\rm cb}A(G) \odot M^{\rm cb}A(G)$ dense in $M^{\rm cb}\ahg$? 
\end{question}

Given $w\in \cbmh$,   let $R_w$ be the dual of $\bm_w$; clearly, $R_w$ is a completely
bounded weak* continuous map on $\ev(G)$ and $\|R_w\|_{\cb} = \|w\|_{\rm cbm}$.

The proof of the following proposition is  similar to the proof of   \cite[Theorem~1.6]{ch}
which characterises the completely bounded multipliers of Fourier algebras of locally compact groups.
We note that, if $H$ is a finite group, then $A(H)$ coincides, as a set, with the algebra of all complex valued
functions on $H$, and hence the operator projective tensor product $A_{\rm h}(G)\hat{\otimes} A(H)$ can be identified
in a natural fashion with a space of functions on $G\times G\times H$.

\begin{proposition}\label{p_McbAh(G)-finite-groups}
Let $u$ be a bounded continuous function on $G \times G$.
The following are equivalent:

(i) \ $u \in M^{\rm cb}\ahg$;

(ii)  there exists $C>0$ such that for every finite group $H$,
$u \otimes 1$  belongs to $M( \ahg \hat\otimes A(H))$ and $\|u \otimes 1\|_{M( \ahg \hat\otimes A(H))} \leq C$.
\end{proposition}
\begin{proof}
Suppose that $H$ is a finite group and
let $k_1,\ldots, k_n\in \bb{N}$ be the dimensions of the (pairwise inequivalent) irreducible representations of $H$.
Then $\vn(H) \cong \bigoplus_{i=1}^n  {\M}_{k_i}$.
Up to complete isometries, by  Corollary~7.1.5 and equation (7.1.16)  in \cite{er_book}, we have 
\begin{eqnarray}\label{eq_ahfin}
(\ahg \hat{\otimes} A(H))^* &=& CB(\ahg , A(H)^*) = CB(\ahg , \bigoplus_{i=1}^n  {\M}_{k_i}) \nonumber \\
& = & \bigoplus_{i=1}^n  CB(\ahg, {\M}_{k_i}) = \bigoplus_{i=1}^n   {\M}_{k_i}({\ahg}^*) \\
&=&  \bigoplus_{i=1}^n   {\M}_{k_i}(\ev(G))\nonumber.
\end{eqnarray}

(ii)$\Rightarrow$(i)
Suppose that $u$ is a bounded continuous function satisfying the condition in (ii).
For a fixed positive integer $l$, choose $H$ so that for some $i_0$, $k_{i_0} = l$.
By restricting $R_{u\otimes 1}$ to the $i_0$th component of
$(\ahg  \hat\otimes A(H))^*$ in the decomposition (\ref{eq_ahfin}),
we get
\[
\| R_u \otimes {\id}_{k_{i_0}}\| \leq \| \bm^*_{u\otimes 1}\|_{\mathcal{B}(\ev(G)  {\otimes} \vn(H))} = \| u\otimes 1\|_{M( \ahg \hat\otimes A(H))}\leq C.
\]
It follows that $R_u$ is a completely bounded weakly$^*$ continuous map on $\ev(G)$;
consequently, $u \in M^{\rm cb}\ahg$.

(i)$\Rightarrow$(ii) follows from the identification (\ref{eq_ahfin}).
\end{proof}

In the sequel, for $w\in \mh$ and $u\in \ev(G)$,
we often write $w\cdot u = R_w(u)$.
It is clear that
\begin{equation}\label{eq_iw}
\|w\cdot u\|_{\eh} \leq \|w\|_{\rm cbm}\|u\|_{\eh}.
\end{equation}
Note that if $w\in \cbmh$ then
\begin{equation}\label{eq_onel}
w\cdot (\lambda_s\otimes\lambda_t) = w(s,t)(\lambda_s\otimes \lambda_t), \ \ \ s,t\in G.
\end{equation}
Indeed, if $v\in \ahg$ then, by Lemma \ref{l_iotah},
$$\langle w\cdot (\lambda_s\otimes\lambda_t),v\rangle
= \langle \lambda_s\otimes\lambda_t,wv\rangle
= (wv)(s,t) = \langle w(s,t) (\lambda_s\otimes \lambda_t), v\rangle,$$
and (\ref{eq_onel}) is proved. Since
$\|\lambda_s\otimes\lambda_t\|_{\eh} = 1$,
we have that
\begin{equation}\label{eq_unif}
|w(s,t)|\leq \|w\|_{\rm cbm}, \ \ \ s,t\in G.
\end{equation}

In the next proposition, we equip $M^{\rm cb}\ahg$ with the operator space
structure arising from its inclusion into $CB(\ahg)$.

\begin{proposition}\label{p_cow} 
The map
\[
M^{\rm cb}\ahg \times \ev(G) \to \ev(G)
\]
given by $(w,u)\to w\cdot u$ turns $\ev(G)$ into a completely contractive
operator  $M^{\rm cb}\ahg$-module.
Moreover, the module action is weak* continuous with respect to the second variable.
\end{proposition}
\begin{proof}
The map $\theta : (w,u)\to w\cdot u$ is clearly bilinear. If $v\in \ahg$, $w_1,w_2\in M^{\rm cb}\ahg$,
and $u\in \ev(G)$ then
$$\langle (w_1w_2)\cdot u, v\rangle = \langle u, w_1w_2 v\rangle
= \langle u, w_2w_1 v\rangle = \langle w_2\cdot u, w_1 v\rangle = \langle w_1\cdot (w_2\cdot u), v\rangle,$$
and hence the map $(w,u)\to w\cdot u$ is a module action.

For $[\alpha_{p,q}] \in \M_n(\ev(G))$ and $[\beta_{l,m}] \in \M_r(\ahg)$, let
\[
\langle \langle [\alpha_{p,q}],  [\beta_{l,m}]\rangle\rangle = [\langle \alpha_{p,q},\beta_{l,m}\rangle]\in \M_{nr}.
\]
Suppose  that $[u_{i,j}] \in \M_n(\ev(G))$ and
$[w_{p,q}] \in \M_k(M^{\rm cb}A_{\rm h}(G))$.  Then
\begin{eqnarray*}
& & \|\theta^{(n,k)}( [w_{p,q}], [u_{i,j}] ) \| \\
& = & \sup\{\|\langle\langle  \theta^{(k,n)} ( [w_{p,q}], [u_{i,j}]) , [v_{s,t}]\rangle\rangle\| : [v_{s,t}]\in \ball(\M_r(\ahg), \ r\in \Bbb{N}\}\\
& = & \sup\{\|\langle\langle     [u_{i,j}]  , [w_{p,q} v_{s,t}]\rangle\rangle\| : [v_{s,t}]\in \ball(\M_r(\ahg), \ r\in \Bbb{N}\}\\
&\leq & \| [u_{i,j}]\|_{{\rm eh}} \sup\{\|  [w_{p,q} v_{s,t}]\|_{{\rm h}} : [v_{s,t}] \in \ball(\M_r(\ahg), \ r\in \Bbb{N}\}.
\end{eqnarray*}
For a fixed $r \in \Bbb{N}$,
let $T_{[w_{p,q}]} : \ahg\to \M_{k}(\ahg)$ be the operator given by
$T_{[w_{p,q}]}(v) = [w_{p,q}v]$, $v\in \ahg$.
By the definition of the operator space structure of $M^{\rm cb}\ahg$,
for each $[v_{s,t}]$   in the unit ball of $\M_r(\ahg)$ we have
\[
\|[ T_{[w_{p,q}]}(v_{s,t})]\|_{{\rm h}} \leq \| [w_{p,q}]  \|_{\M_k(M^{\rm cb}\ahg)}.
\]
This implies that
\[
\|\theta^{(k,n)} ( [w_{p,q}], [u_{i,j}]) \|_{\eh} \leq \| [w_{p,q}] \|_{\M_k(M^{\rm cb}\ahg)} \| [u_{i,j}]\|_{{\rm eh}}.
\]

It remains to show that the module action is weak* continuous with respect to the second variable.
To this end, let $(u_i)_i\subseteq \ev(G)$ be a net converging in the weak* topology to $u\in \ev(G)$.
If $w\in M^{\cb}\ahg$ and $v\in \ahg$ then
$$\langle w\cdot u_i, v\rangle = \langle  u_i, w v\rangle \to_i \langle  u, w v\rangle
= \langle w\cdot u, v\rangle,$$
establishing the claim.
\end{proof}

\begin{lemma}\label{l_elt}
Let $\psi_1,\psi_2\in A(G)$ and $w = \psi_1\otimes \psi_2$. If $T_1,T_2\in \vn(G)$ then
$w\cdot (T_1\otimes T_2) = (\psi_1\cdot T_1) \otimes (\psi_2\cdot T_2)$.
\end{lemma}
\begin{proof}
Whenever $\phi_1,\phi_2\in A(G)$, we have
\begin{eqnarray*}
\langle w\cdot (T_1\otimes T_2), \phi_1\otimes\phi_2\rangle
& = & \langle T_1\otimes T_2, \psi_1\phi_1\otimes \psi_2\phi_2\rangle
=  \langle T_1, \psi_1\phi_1\rangle \langle T_2, \psi_2\phi_2\rangle\\
& = & \langle \psi_1\cdot T_1,\phi_1\rangle \langle \psi_2\cdot T_2, \phi_2\rangle\\
& = & \langle (\psi_1\cdot T_1) \otimes (\psi_2\cdot T_2), \phi_1\otimes\phi_2\rangle.
\end{eqnarray*}
Since both $w\cdot (T_1\otimes T_2)$ and $(\psi_1\cdot T_1) \otimes (\psi_2\cdot T_2)$
are bounded functionals on $\ahg$ and
$A(G)\odot A(G)$ is dense in $\ahg$, we conclude that
$w\cdot (T_1\otimes T_2) = (\psi_1\cdot T_1) \otimes (\psi_2\cdot T_2)$.
\end{proof}

Recall that  $\hat{\iota}^* : \ev(G) \to \vn(G\times G)$,  the dual of
$\hat{\iota} : A(G\times G)\to \ahg$,
is completely contractive and injective (see \cite[Corollary 3.8]{bs}).

\begin{lemma}\label{l_inter}
Let $u\in \ev(G)$ and $w\in  A(G \times G)$.
Then   $\hat{\iota}^*(\hat{\iota}(w)\cdot u) = w\cdot \hat{\iota}^*(u)$.
\end{lemma}
\begin{proof}
 For every $v\in A(G\times G)$, we have
$$\langle \hat{\iota}^*(\hat{\iota}(w)\cdot u), v\rangle =
\langle \hat{\iota}(w)\cdot u, \hat{\iota}(v)\rangle  = \langle u, \hat{\iota}(w) \hat{\iota}(v)\rangle
= \langle u, \hat{\iota}(w v)\rangle
= \langle w\cdot \hat{\iota}^*(u), v\rangle.$$
\end{proof}

We recall that, if $C > 0$, a locally compact group $G$ is called \emph{weakly amenable}
with constant $C$ \cite{coh}, if there exists
a net $(\phi_{\alpha})_{\alpha}$ of compactly supported elements of $A(G)$
such that $\|\phi_{\alpha}\|_{\rm cbm}\leq C$ for all $\alpha$ and $\phi_{\alpha}\to 1$ uniformly on compact sets.

\begin{theorem}\label{p_HWA}
Let $G$ be a locally compact group and $C>0$. The  following are equivalent:

(i) \ \ $G$ is weakly amenable with constant $C$;

(ii) \ there exists a net $(w_{\alpha})_{\alpha}$
of compactly supported elements of $A(G)\odot  A(G)$ such that
$\|w_{\alpha}\|_{\rm cbm}\leq C$ for all $\alpha$ and
$w_{\alpha}v\to v$ in $\ahg$ for every $v\in \ahg$;

(iii) there exists a net $(w_{\alpha})_{\alpha}$
of compactly supported elements of $M^{\cb}\ahg$ such that
$\|w_{\alpha}\|_{\rm cbm}\leq C$ for all $\alpha$ and
$w_{\alpha}v\to v$ in $\ahg$ for every $v\in \ahg$.
\end{theorem}
\begin{proof}
(i)$\Rightarrow$(ii)
Suppose $G$ is weakly amenable. By \cite[Proposition~1.1]{coh},
there exist $C > 0$ and a net $(\phi_{\alpha})_{\alpha}\subseteq  A(G)$
of compactly supported elements such that $\|\phi_{\alpha}\|_{\rm cbm}\leq C$ for all $\alpha$ and
$\phi_{\alpha}\phi\to \phi$ in $A(G)$ for all $\phi\in A(G)$.
Set $w_{\alpha} = \phi_{\alpha}\otimes \phi_{\alpha}$.
If $\psi_1,\psi_2\in A(G)$ then, clearly,
$w_{\alpha}(\psi_1\otimes \psi_2)\to_{\alpha} \psi_1\otimes \psi_2$ in $\ahg$.
If $v\in \ahg$ is arbitrary and $\epsilon > 0$, fix $v_0 \in A(G)\odot A(G)$
so that $\|v - v_0\|_{\hh} < \epsilon/3C$.
Let $\alpha_0$ be such that
$\|w_{\alpha} v_0 - v_0\|_{\hh} < \epsilon/3$ for all $\alpha \geq \alpha_0$.
If $\alpha \geq \alpha_0$ then
$$\|w_{\alpha} v - v\|_{\hh}  \leq
\|w_{\alpha} v - w_{\alpha} v_0\|_{\hh}
+ \|w_{\alpha} v_0 - v_0\|_{\hh}
+ \|v_0 - v\|_{\hh} \leq \epsilon.$$

(ii)$\Rightarrow$(iii) follows from Proposition \ref{p_algtenin} and the fact that $A(G)\subseteq M^{\cb}A(G)$.

(iii)$\Rightarrow$(i) Suppose that $(w_{\alpha})_{\alpha}$ is a net
of compactly supported elements of $M^{\cb}\ahg$ such that
$\|w_{\alpha}\|_{\rm cbm}\leq C$ for all $\alpha$ and
$w_{\alpha}v\to v$ in $\ahg$ for every $v\in \ahg$.
By Proposition \ref{p_ccehag} (i), $\ahg$ is a regular Banach algebra;
it follows that $(w_\alpha)_\alpha \subseteq \ahg$.

Note that $\|R_{w_{\alpha}}\|_{\cb}\leq C$ for each $\alpha$
and $R_{w_\alpha}(T) \rightarrow T$ in the weak* topology
of $\ev(G)$, for every $T \in \ev(G)$.
Let $\Psi: \vn(G) \rightarrow \ev(G)$ be the map given by
$\Psi(T) = T \otimes I$.
Clearly, $\Psi$ is weak* continuous and completely contractive;
in fact, $\Psi$ is the dual of the map
$\id  \otimes \delta_e$.
Note, in addition, that the multiplication map
$m : T\otimes S \mapsto TS$
extends uniquely to a weak* continuous completely contractive map  from
$\vn_{\sigma {\rm h}}(G)$ onto $\vn(G)$ (see \cite[p. 133]{er}).
We denote again by $m$ its restriction to a map from $\ev(G)$ into $\vn(G)$ 
(see \cite[Theorem 5.7]{er}).
Clearly, $(m \circ \Psi)(T) = T$ for every $T\in \vn(G)$.
We thus have that
$(m\circ R_{w_\alpha} \circ \Psi)_\alpha$ is a net of weak* continuous
maps on $\vn(G)$
whose completely bounded norm is uniformly bounded by $C$.
Moreover,
$$(m \circ R_{w_\alpha} \circ \Psi)(\lambda_s) = w_\alpha(s,e) \lambda_s,$$
for each $s\in G$.
Let $\psi_\alpha : G\to \bb{C}$ be the function given by $\psi_\alpha(s) = w_\alpha(s,e)$.
Assuming that $\supp(w_{\alpha})\subseteq K_{\alpha}\times K_{\alpha}$ for some
compact subset $K_{\alpha}\subseteq G$,
let $\phi_{\alpha}\in A(G)$ be a compactly supported function taking the value $1$ on
$K_{\alpha}$ and at $e$. Then $\phi_{\alpha}\otimes \phi_{\alpha} \in A_{\hh}(G)$ and hence
$w_{\alpha}(\phi_{\alpha}\otimes \phi_{\alpha})\in A_{\hh}(G)$.
It follows that
$$\psi_\alpha = (\id\otimes\delta_e)(w_{\alpha}(\phi_{\alpha}\otimes \phi_{\alpha})) \in A(G).$$

For each $T \in \vn(G)$ and $\psi \in A(G)$, we have
\begin{eqnarray*}
\langle \psi_\alpha \psi  - \psi , T\rangle &=& \langle \psi, \psi_\alpha \cdot T -T\rangle \\
&=& \langle \psi, (m \circ R_{w_\alpha} \circ \Psi)(T) - T\rangle\\
&=& \langle \psi, m \circ (R_{w_\alpha} - \id)\circ \Psi(T)\rangle \\
&=& \langle m_*(\psi), (R_{w_\alpha} - \id)\circ \Psi(T)\rangle \rightarrow\mbox{}_{\alpha} \  0.
\end{eqnarray*}
Therefore, $\psi_\alpha \psi \rightarrow \psi$ in the weak topology of $A(G)$.
Thus, $\psi$ belongs to the weak closure of the convex
hull of the set $\{\psi_\alpha \psi\}_\alpha$.

Fix $0\neq \psi \in A(G)$. Since the weak closure and the norm closure of a convex set are equal,
the previous paragraph
implies the existence of a net $(\psi'_\beta)_\beta$ in $A(G)$
(depending on $\psi$) with $\sup_\beta \|{\psi'_\beta}\|_{\rm cbm} \leq C$ and
 \[
 \|\psi'_\beta \psi - \psi\|_{\rm cbm} \leq \|\psi'_\beta \psi - \psi\|_{A} \rightarrow\mbox{}_{\beta} \ 0.
 \]
Consequently,  the normed algebra $(A(G), \|\cdot\|_{\rm cbm})$
has an approximate unit, bounded in $\|\cdot\|_{\rm cbm}$ by $C$.
By \cite[Theorem~1]{w}, $(A(G), \|\cdot\|_{\rm cbm})$ has a
bounded approximate identity, and the weak amenability of $G$ follows.
\end{proof}

\begin{remark}\label{r:in-A(G)xA(G)}
{\rm By the proof of Theorem~\ref{p_HWA}, if
condition (iii) is satisfied then then the net
$(w_\alpha)_\alpha$ can be chosen of the form of $\phi_\alpha \otimes \phi_\alpha$ for
a net $(\phi_\alpha)_\alpha$  of compactly supported elements  of $A(G)$.}
\end{remark}

In the remainder of the section,
we will be concerned with the extended Haagerup tensor product
$A_{\eh}(G):=A(G)\otimes_{\eh} A(G)$ and its connection with $\ahg$ and $M^{\cb}\ahg$.
We will use some technical notions from \cite{er} and we refer the reader to the latter paper for details.
Set
$$\vn\mbox{}_{\sigma \hh}(G) := \vn(G)\otimes\mbox{}_{\sigma \hh} \vn(G);$$
we have the canonical identification \cite{er}
$$A_{\eh}(G)^* \equiv \vn\mbox{}_{\sigma \hh}(G).$$
Similarly to the elements of $\vn_{\eh}(G)$, every element $w$ of $A_{\eh}(G)$
has a representation $w = \phi\odot \psi := \sum_{i=1}^{\infty} \phi_i\otimes\psi_i$,
where $\phi = (\phi_i)_{i\in \bb{N}}$ (resp. $\psi = (\psi_i)_{i\in \bb{N}}$)
is a bounded row (resp. column) with entries in $A(G)$.
Recalling the identification
$$A_{\eh}(G) \equiv CB_m^\sigma(\vn(G)\times \vn(G),\bb{C}),$$
where the latter space consists of all multiplicatively bounded separately weak* continuous bilinear functionals on
$\vn(G)\times \vn(G)$ \cite{er}, with a given $\omega \in A_{\eh}(G)$,
we associate the function $w_{\omega} : G\times G\to \bb{C}$ with
$$w_{\omega}(s,t)
= \langle\omega,\lambda_s\otimes\lambda_t\rangle = \omega(\lambda_s,\lambda_t)
=  \sum_{i=1}^{\infty} \langle \phi_i,\lambda_s\rangle \langle \psi_i,\lambda_t\rangle = \sum_{i=1}^{\infty} \phi_i(s) \psi_i(t).$$
By \cite{er}, $A_{\eh} (G)$ is a completely contractive Banach
algebra and the multiplication is defined as the composition of the following maps:
\begin{eqnarray*}
A_{\eh}(G)\hat\otimes A_{\eh}(G)
& \stackrel{\Psi}\longrightarrow & A_{\eh}(G)\otimes_{\rm nuc} A_{\eh}(G)\\
& \stackrel{S_e}\longrightarrow &
(A(G)\hat\otimes A(G))\otimes_{\eh}(A(G)\hat\otimes A(G))\stackrel{m_A\otimes_{\eh}m_A}\longrightarrow A_{\eh}(G),
\end{eqnarray*}
where $\Psi$ is the canonical complete contraction from the projective tensor product to
the nuclear tensor product of two copies of the operator space $A_{\eh}(G)$ (see \cite[p. 139]{er}),
$S_e$ is the shuffle map (see \cite[Theorem 6.1]{er}) and
$m_A$ is the multiplication in $A(G)$.
By \cite[Theorem 6.1]{er},  $S_e^* = S_\sigma$, where $S_\sigma$ is the shuffle map
$$(\vn(G)\bar\otimes \vn(G))\otimes\mbox{}_{\sigma\hh}(\vn(G)\bar\otimes \vn(G))
\to \vn\mbox{}_{\sigma\hh}(G) \bar\otimes \vn\mbox{}_{\sigma\hh}(G)$$
defined on the elementary tensors by
$$S_\sigma((S_1\otimes S_2)\otimes(T_1\otimes T_2))=(S_1\otimes T_1)\otimes (S_2\otimes T_2).$$

Note that $m_A\otimes_{\eh}m_A$ is defined as the restriction to the space
$$A(G\times G)\otimes_{\eh} A(G\times G) = (A(G)\hat\otimes A(G))\otimes_{\eh}(A(G)\hat\otimes A(G))$$
of the map
$$(m_A^*\otimes_{\hh} m_A^*)^*: (\vn(G\times G)\otimes_{\hh}\vn(G\times G))^*\to (\vn(G)\otimes_{\hh} \vn(G))^*.$$
Thus, for $\omega_1$, $\omega_2\in A_{\eh}(G)$ we have
\begin{eqnarray*}
&&\langle \omega_1\cdot\omega_2,\lambda_s\otimes\lambda_t\rangle_{A_{\eh}(G), \vn_{\sigma\hh}(G)}\\
&&=\langle (m_A\otimes_{\eh} m_A)\circ S_e\circ\Psi(\omega_1\otimes\omega_2),
\lambda_s\otimes\lambda_t\rangle_{A_{\eh}(G), \vn_{\sigma\hh}(G)}\\
&&=\langle S_e\circ\Psi(\omega_1\otimes\omega_2),m_A^*(\lambda_s)\otimes m_A^*(\lambda_t)\rangle_{A_{\eh}(G\times G),
\vn_{\sigma\hh}(G\times G)}\\
&&=\langle S_e\circ\Psi(\omega_1\otimes\omega_2),\lambda_s\otimes\lambda_s\otimes \lambda_t\otimes\lambda_t\rangle_{A_{\eh}(G\times G), \vn_{\sigma\hh}(G\times G)}\\
&&=\langle\Psi(\omega_1\otimes\omega_2),
(\lambda_s\otimes\lambda_t)\otimes(\lambda_s\otimes\lambda_t)\rangle_{A_{\eh}(G)\otimes_{\rm nuc} A_{\eh}(G), \vn_{\sigma\hh}(G)\bar\otimes \vn_{\sigma\hh}(G) }\\
&&=\langle\omega_1,\lambda_s\otimes\lambda_t\rangle_{A_{\eh}(G), \vn_{\sigma\hh}(G)}\langle\omega_2,\lambda_s\otimes\lambda_t\rangle_{A_{\eh}(G), \vn_{\sigma\hh}(G)},
\end{eqnarray*}
giving
\begin{equation}
w_{\omega_1\cdot\omega_2}(s,t) = w_{\omega_1}(s,t) w_{\omega_2}(s,t).
\end{equation}
This shows that the map $\omega \to w_{\omega}$ from $A_{\eh}(G)$ into the algebra
of all separately continuous functions on $G\times G$ is a homomorphism.
Since the elementary tensors $\lambda_s\otimes\lambda_t$ span a weak* dense subspace of
$\vn_{\sigma\hh}(G)$ \cite[Lemma 5.8]{er}, we have that the latter map is injective.
This allows us to view $A_{\eh}(G)$ as an algebra (with respect to pointwise multiplication) of
(separately continuous) functions on $G\times G$.

The operator multiplication in $\vn(G)$
can be extended uniquely to a weak* continuous completely contractive map
$m : \vn(G) \otimes_{\sigma\hh} \vn(G) \to \vn(G)$ (see \cite[Proposition~5.9]{er}).
Following M. Daws \cite{daws}, we denote by $m_*$ its predual; thus, $m_*$
is a complete contraction from $A(G)$ into $A(G)\otimes_{\eh} A(G)$.
The following special case of \cite[Theorem~9.2]{daws} combined with the
remarks after its proof, will play a crucial role
in the next section.

\begin{theorem}[\cite{daws}]\label{th_daws}
The range of $m_*$ is in $M^{\cb}\ahg$ and $m_*$ is a complete contraction
when considered as a map from $A(G)$ to $M^{\cb}\ahg$.
\end{theorem}

We note that
$$m_*(\phi)(s,t) = \langle m_*(\phi), \lambda_s\otimes\lambda_t\rangle
= \langle \phi, \lambda_{st}\rangle = \phi(st),$$
for all $\phi\in A(G)$ and all $s,t\in G$.


\section{Spectral synthesis in $\ahg$}\label{s_ss}

By Proposition \ref{p_ccehag}, $\ahg$ is a regular 
commutative semisimple Banach algebra with Gelfand spectrum $G\times G$,
and thus the problem of spectral synthesis for closed subsets of $G\times G$
is well-posed.
In this section, we link this problem to the problem of spectral synthesis in $A(G)$.
We start by recalling some definitions, which will be specialised to $\ahg$ and $A(G)$ in the sequel.
Suppose that
$\cl A$ is a regular commutative semisimple Banach algebra with Gelfand spectrum $\Omega$; we
can thus identify $\cl A$ with a subalgebra of $C_0(\Omega)$.
Given a subset $\cl J\subseteq \cl A$, we let
$${\rm null}(\cl J) = \{x\in \Omega : a(x) = 0 \mbox{ for all } a\in \cl J\}$$
be its null set. Given a closed subset $E\subseteq \Omega$,
let
$$I_{\cl A}(E) = \{a\in \cl A : a(x) = 0 \mbox{ for all } x\in E\},$$
 \[
 I_{\mathcal{A}}^c(E)=\{ a \in I_{\mathcal{A}}(E): \; \text{$a$ has compact support}\},
 \]
and
$$J_{\cl A}(E) = \overline{\{a\in \cl A :  a \mbox{ has compact support disjoint from } E\}}.$$
If $\cl J\subseteq \cl A$ is a closed ideal, then
${\rm null}(\cl J) = E$ if and only if $J_{\cl A}(E)\subseteq \cl J\subseteq I_{\cl A}(E)$
(see {\it e.g.} \cite{hr1}).
The set $E$ is called a set of \emph{spectral synthesis} (resp. \emph{local spectral synthesis}) for $\cl A$, if $I_{\cl A}(E) = J_{\cl A}(E)$ (resp. $\overline{I_{\mathcal{A}}^c(E)} = J_{\mathcal{A}}(E)$).
Equivalently, $E$ is a set of spectral synthesis if $J_{\cl A}(E)^{\perp} = I_{\cl A}(E)^{\perp}$
where, for a subset $\cl J\subseteq\cl A$, we have set
$$\cl J^{\perp} = \{\tau\in \cl A^* : \tau(a) = 0, \mbox{ for all } a\in \cl J\}$$
to be the annihilator of $\cl J$ in $\cl A^*$.

For an element $\tau \in\cl A^*$, following \cite[Definition~5.1.12]{kaniuth}, we set
$$\supp\mbox{}_{\cl A} (\tau) \stackrel{def}{=}
\{x\in \Omega : \mbox{ for all open } V\subseteq X \mbox{ with } x\in V
\mbox{ there exists } a\in \cl A$$ $$ \mbox{ with } \supp(a)\subseteq V
\mbox{ such that }
\langle \tau,a\rangle \neq 0\}.$$
Clearly,
\begin{equation}\label{eq_news}
(\supp\mbox{}_{\cl A} (\tau))^c =
\{x\in \Omega : \exists \mbox{ an open set } V\subseteq \Omega \mbox{ with } x\in V
\mbox{ such that}
\end{equation}
$$ \mbox{ if } a\in \cl A \mbox{ and } \supp(a)\subseteq V
\mbox{ then }
\langle \tau,a\rangle = 0\}.$$
It is easy to see that $E$ is a set of spectral synthesis if and only if
$\langle \tau, a\rangle = 0$ for all $a\in I_{\cl A}(E)$ and all $\tau \in \cl A^*$
with $\supp_{\cl A}(\tau)\subseteq E$.

For $T\in \vn(G)$, we set $\supp_G(T) \stackrel{def}{=} \supp_{A(G)}(T)$ to be the support of $T$
introduced by Eymard \cite{eymard}, and for
$u\in \ev(G)$, we set $\supp_{\hh}(u) \stackrel{def}{=} \supp_{\ahg}(u)$.
In the latter case, there is another natural candidate for a support of $u$, namely, the set
$\supp_{G\times G}(u) \stackrel{def}{=} \supp_{G\times G}(\hat{\iota}^*(u))$ where $\hat{\iota}: A(G \times G)  \rightarrow \ahg$ is the complete contraction defined in (\ref{eq_projective_to_Ah}).
In the next lemma we show that these two concepts coincide.

\begin{lemma}\label{l_supp-equality}
Let $u \in \ev(G)$ and $w \in \ahg$. Then

(i)\ \ \ $\supp_{\hh}(u)= \supp_{G\times G}(u)$;

(ii)  \ if $\supp_{\hh}(u) = \emptyset$ then $u = 0$;

(iii)\   $\supp_{\hh}(w\cdot u) \subseteq \supp(w)\cap \supp_{\hh}(u)$.
\end{lemma}
\begin{proof}
(i)
Suppose that $x\in \supp_{G\times G} (u)$ and let $V\subseteq G\times G$ be an open
neighbourhood of $x$. Then there exists $w\in A(G\times G)$ such that  $\supp(w) \subseteq V$ and
$\langle \hat{\iota}^*(u),w\rangle \neq 0$.
Thus, $\langle u,\hat{\iota}(w)\rangle \neq 0$ and, since $\hat{\iota}(w)$ is supported in $V$,
we have that $x\in \supp_{\hh}(u)$.

An argument similar to the one in the proof of  \cite[Proposition~4.4]{eymard} 
shows that $\supp_{\hh}(u)$  is the set of all $x\in G\times G$ so that, if
$w \in A_{\hh}(G)$ is such that $w\cdot u =0$, then $w(x)=0$.
Let $x\in \supp_{\hh}(u)$
and $w\in A(G\times G)$ with $w\cdot \hat{\iota}^*(u) = 0$. By Lemma~\ref{l_inter},
$\hat{\iota}(w) \cdot u = 0$ and so $\hat{\iota}(w)(x) = 0$.
By \cite[Proposition~4.4]{eymard}, $x\in \supp_{G \times G}(u)$.

(ii) By (i) and the definition of $\supp_{G \times G}(u)$, the element $\hat{\iota}^*(u)$ of
$\vn(G\times G)$ has empty support. By \cite[Proposition 4.6]{eymard},
$\hat{\iota}^*(u) = 0$ and, since $\hat{\iota}^*$ is injective, $u = 0$.

(iii)    Suppose there exists $x\in \supp_{\hh}(w \cdot u)$ not contained in $\supp(w)$.
Let $V$ be an open neighbourhood of $x$ such that $V \cap \supp(w) = \emptyset$.
There exists $v \in A_{\hh}(G)$ such that $\supp(v) \subseteq V$ and
$\langle w\cdot u, v\rangle \neq 0$.
However, $\langle w\cdot u, v\rangle = \langle u, vw\rangle = 0$ as $vw=0$,
a contradiction. Therefore, $\supp_{\hh}(w\cdot u) \subseteq \supp(w)$.

Finally, suppose that 
$x\in \supp_{\hh}(w \cdot u)$ does not belong to $\supp_{\hh}(u)$,
and let $V\subseteq G\times G$ be a neighbourhood of $x$ with $V\cap \supp_{\hh}(u) = \emptyset$,
satisfying the condition on the right hand side of (\ref{eq_news}).
If $v\in \ahg$ is supported in $V$ then so is $wv$.
By (\ref{eq_news}),
$$\langle w\cdot u,v\rangle = \langle u,wv\rangle = 0,$$
showing that $x\not\in \supp_{\hh}(w \cdot u)$, a contradiction.
\end{proof}

For a subset $E\subseteq G$, let
$$E^{\sharp} = \{(s,t)\in G\times G : st\in E\}.$$
The proof of the following lemma is immediate and we omit it.

\begin{lemma}\label{l:sharp-sets}
Let $E$ be a subset of $G$. Then $\overline{(E^\sharp)}=(\overline{E})^\sharp$ and $(E^c)^\sharp = (E^\sharp)^c$.
\end{lemma}

If $w : G\times G\to \bb{C}$ and
$s,t\in G$, we let $w_{s,t} : G\to \bb{C}$ be the function
given by $w_{s,t}(r) = w(sr,r^{-1}t)$.
For $x\in G$, we let $w_x : G\times G\to \bb{C}$ be the function given by
$w_x(s,t) := w_{s,t}(x) = w(sx, x^{-1}t)$, and
$\hat{w}_x : G\to \bb{C}$ be the function given by $\hat{w}_x(t) := w_x(e,t)=w(x,x^{-1}t)$.

\begin{lemma}\label{l_newinde}
Let $w\in \ahg$. Then

(i) \ \ for each $r\in G$, the function $w_r$ belongs to $\ahg$ and $\|w_r\|_{\hh} = \|w\|_{\hh}$;

(ii) \ the map $r\to w_r$ from $G$ into $\ahg$ is continuous;

(iii) for each $r\in G$, the function $\hat{w}_r$ belongs to $A(G)$ and $\|\hat{w}_r\|_A \leq \|w\|_{\hh}$;

(iv) \ the map $r\to \hat{w}_r$ from $G$ into $A(G)$ is  continuous.
\end{lemma}
\begin{proof}
Fix, throughout the proof, $w\in \ahg$.
For $r\in G$ and $\phi\in A(G)$, let
$L_r, R_r : A(G)\to A(G)$ be the operators given by
$L_r(\phi)(t) = \phi(r^{-1}t)$ and $R_r(\phi)(s) = \phi(sr)$.
It is well-known that
$L_r$ and $R_r$ are complete isometries;
thus, the operator
$R_r\otimes L_r : \ahg\to \ahg$ is a (complete) isometry.

(i), (ii) Note that
\begin{equation}\label{eq_wrlrrr}
(R_r\otimes L_r)(w) = w_r.
\end{equation}
Indeed, this identity is straightforward
if $w$ is an element of the algebraic tensor product $A(G)\odot A(G)$, 
and, by the density of $A(G)\odot A(G)$ in $\ahg$ and the fact that $\|\cdot\|_{\hh}$
dominates the uniform norm, it holds for an arbitrary $w\in \ahg$.

It is easy to see that the maps
$ r\mapsto R_r(\phi)$ and $r \mapsto L_r(\phi)$ from $G$ into $A(G)$ are continuous, for every $\phi \in A(G)$.
By (\ref{eq_wrlrrr}), the map $r \mapsto w_r$ from $G$ into $\ahg$ is continuous,
and $\|w_r\|_{\hh} = \|w\|_{\hh}$ for every $r\in G$.

(iii)
Similarly to (\ref{eq_wrlrrr}), one can show that
$\hat{w}_r = (\delta_r\otimes L_r)(w)$ where $\delta_r$ denotes the evaluation at $r$.
The map $\delta_r\otimes L_r : \ahg \to A(G)$ is completely contractive;
it follows that $\hat{w}_r \in A(G)$ and
$\|\hat{w}_r\|_A\leq \|w\|_{\hh}$, $r\in G$.

(iv) Fix $s\in G$ and $\epsilon > 0$.
Assume that $w = \phi\otimes \psi$, where $\phi, \psi\in A(G)$
and $\psi(t)= (\lambda_t \xi, \eta)$, $t\in G$, for some $\xi,\eta \in L^2(G)$.
\begin{eqnarray*}
\hat{w}_r(t) - \hat{w}_s(t) & = & (\phi(r) - \phi(s)) \psi(r^{-1}t) + \phi(s) (\psi(r^{-1}t) - \psi(s^{-1}t))\\
& = & (\phi(r) - \phi(s)) L_r(\psi)(t) + \phi(s) (L_r(\psi)(t) - L_s(\psi)(t)),
\end{eqnarray*}
for all $t\in G$; thus,
\begin{equation}\label{eq_vrvs}
\hat{w}_r - \hat{w}_s = (\phi(r) - \phi(s)) L_r(\psi) + \phi(s) (L_r(\psi) - L_s(\psi)).
\end{equation}
Let $V_s$ be a neighbourhood of $s$ such that
$|\phi(r) - \phi(s)| < \epsilon$ and $\|\lambda_r \eta - \lambda_s \eta\|_2 < \epsilon$ for all $r \in V_s$.
By (\ref{eq_vrvs}),
\begin{eqnarray*}
\|\hat{w}_r - \hat{w}_s\|_{A} & \leq & |\phi(r)- \phi(s)| \|\psi\|_A + |\phi(s)| \|\xi\|_2 \| (\lambda_r - \lambda_s) \eta\|_2\\
& < & \epsilon \|\psi\|_{A} + \epsilon |\phi(s)| \|\xi\|_2,
\end{eqnarray*}
for all $r \in V_s$.
It follows that, if $w =\sum_{i=1}^n \phi_i \otimes \psi_i$ for $\phi_i, \psi_i \in A(G)$, $i=1, \ldots, n$,
then there exists a neighbourhood $V_s$ of $s$, so that $\|\hat{w}_r - \hat{w}_s\|_A < \epsilon$ for every $r\in V_s$.

Finally, if $w\in \ahg$ is arbitrary, let $v\in A(G)\odot A(G)$ be such that
$\|w - v\|_{\hh} < \epsilon/3$. Let $V_s$ be a neighbourhood of $s$ such that
$\|\hat{v}_r - \hat{v}_s\|_A < \epsilon$ for every $r\in V_s$.
Then, for every $r\in V_s$, using (iii) we have
\begin{eqnarray*}
\|\hat{w}_r - \hat{w}_s\|_A
& \leq &
\|\hat{w}_r - \hat{v}_r\|_A + \|\hat{v}_r - \hat{v}_s\|_A + \|\hat{v}_s - \hat{w}_s\|_A\\
& \leq &
2\|w - v\|_{\hh} + \|\hat{v}_r - \hat{v}_s\|_A < \epsilon.
\end{eqnarray*}
Thus, the map $r \mapsto \hat{w}_r$ is continuous.
\end{proof}

\begin{lemma}\label{l_Gamma}
Let $a\in A(G)$ be a compactly supported function and
$\cl V_a = \overline{aA(G) \odot  aA(G)}^{\|\cdot\|_{\hh}}$.

(i) \ \ If $v\in \cl V_a$ then the function
$s\to v(s,s^{-1}t)$ is integrable for each $t\in G$.

(ii) \ For $v\in \cl V_a$, the function $\Gamma_a(v) : G\to \bb{C}$ given by
\[
\Gamma_a(v)(t) = \int_G v(s,s^{-1}t)ds, \ \ \ \ t\in G,
\]
is compactly supported and belongs to $A(G)$.

(iii) The map $\Gamma_a : \cl V_a\to A(G)$ is bounded and $\|\Gamma_a\|\leq |\supp(a)|$.
\end{lemma}
\begin{proof}
Set $F = \supp(a)$. Note that, by the
injectivity of the Haagerup tensor product, there is a natural completely
isometric identification $\cl V_a \equiv \overline{aA(G)}\otimes_{\hh} \overline{aA(G)}$.

(i) For $v\in \cl V_a$ and $t\in G$, the function
$s\to v(s,s^{-1}t)$ 
is continuous and supported on the compact set $F\cap tF^{-1}$; it is hence integrable.

(ii), (iii) Let $v\in \cl V_a$.
Since $v$ is supported on $F\times F$, we have that $\hat{v}_s = 0$ if $s\not\in F$.
By Lemma \ref{l_newinde} (iv), the function $s\to \hat{v}_s$ from $G$ into $A(G)$ is continuous.
It follows that the integral
$$\Gamma'(v) = \int_G \hat{v}_s ds$$ is
well-defined in the Bochner sense. By Lemma \ref{l_newinde} (iii),
$\|\Gamma'(v)\|_{A} \leq |F|\|v\|_{\hh}$, and hence $\Gamma'$ is a  bounded map
on $\cl V_a$ with norm not exceeding $|F|$.
If $t\in G$ then
$$\Gamma'(v)(t)
= \left\langle \int_G \hat{v}_s ds, \lambda_t\right\rangle
= \int_G \langle  \hat{v}_s, \lambda_t\rangle ds
= \int_G v(s,s^{-1}t)ds = \Gamma_a(v)(t);$$
thus, $\Gamma' = \Gamma_a$. In particuar, $\Gamma_a$ takes values in $A(G)$ and
is bounded with $\|\Gamma_a\|\leq |F|$.
In addition, $\Gamma_a(v)(t) = 0$ whenever $t\not\in FF$, and so the function $\Gamma_a(v)$ is
compactly supported. 
\end{proof}

\begin{lemma}\label{l_Gammaj}
Let $E\subseteq G$ be a closed set and $a\in A(G)$ be compactly
supported. 
Suppose that $v\in \cl V_a$ has compact support disjoint from $E^{\sharp}$.
Then $\Gamma_a(v)\in  J_{A(G)}(E)$.
\end{lemma}
\begin{proof}
Suppose that $v\in \cl V_a$ has compact support disjoint from $E^{\sharp}$ and set $F = \supp(v)$.
By Lemma \ref{l_Gamma}, $\Gamma_a(v)$ is compactly supported. 
By the continuity of the group multiplication,
the set $\tilde{F}\stackrel{def}{=}\{xy: (x,y) \in F\}$ is compact;
moreover, $\tilde{F} \cap E = \emptyset$.
Therefore, there exists an open subset $W\subseteq G$ such that $\overline{W}$ is compact and
$\tilde{F} \subseteq W \subseteq  \overline{W} \subseteq E^c$.
Again by the continuity of the multiplication,
$W^\sharp$ is an open subset of $G\times G$. Clearly,
\[
F \subseteq \tilde{F}^\sharp \subseteq W^\sharp \subseteq \overline{W}^\sharp.
\]
Now let $t\in \overline{W}^c$; then
$(s,s^{-1}t) \in (\overline{W}^c)^\sharp$ for each $s\in G$.
By Lemma \ref{l:sharp-sets}, $(s,s^{-1}t) \in (\overline{W}^\sharp)^c$ and so
$(s,s^{-1}t) \notin F$, $s\in G$.
Therefore,
\[
\Gamma_a(v)(t) = \int_{G} v(s,s^{-1}t) ds =0.
\]
It follows that $\Gamma_a(v)$ has compact support (within  $\overline{W}$) disjoint from $E$, and hence
$\Gamma_a(v)\in J_{A(G)}(E)$.
\end{proof}

\begin{theorem}\label{t_spectral-of-E-sharp}
Let $E\subseteq G$ be a closed set.
If $E^\sharp$ is a set of  spectral synthesis for $\ahg$ then $E$ is a set of local spectral synthesis for $A(G)$.
Moreover, if $A(G)$ has a (possibly unbounded) approximate identity then $E$ is a set of spectral synthesis.
\end{theorem}
\begin{proof}
Let $\phi \in A(G)$ vanish on $E$.
Then $m_*(\phi)$ vanishes on $E^\sharp$.
By  Theorem~\ref{th_daws}, $m_*(\phi) \in M^{\rm cb}\ahg$.

Fix $v \in \ahg \cap C_c(G \times G)$ and let
$K\subseteq G$ be a compact subset such that $\supp(v) \subseteq K\times K$.
The element $m_*(\phi)v$ of $\ahg$ vanishes on $E^\sharp$;
since $E^\sharp$ is a set of spectral synthesis for $\ahg$, there
exists a sequence $(v_n)_{n\in \bb{N}}\subseteq \ahg$, whose elements have compact support disjoint from $E^\sharp$,
such that $v_n \rightarrow_{n\to\infty} m_*(\phi) v$.

Since $A(G)$ is a regular Banach algebra, there exists a
compactly supported  function $a_K \in A(G)$ which is equal to  $1$ on $K$.
Setting $a_{K \times K}= a_K \otimes a_K$,
we have $v  a_{K \times K} = v$.
After replacing $v_n$ by $v_n a_{K\times K}$ if necessary, we may assume that
$(v_n)_{n\in \bb{N}} \subseteq a_{K \times K} \ahg$.

Note that
\[
\overline{a_{K\times K} \ahg}^{\|\cdot\|_{\hh}}  = \overline{(a_K A(G)) \odot (a_KA(G))}^{\|\cdot\|_{\hh}};
\]
therefore $(v_n)_{n\in \bb{N}} \subseteq \cl V_{a_K}$ and hence $m_*(f) v\in \cl V_{a_K}$.
By Lemma \ref{l_Gamma}, for any $t\in G$ we have
\begin{eqnarray*}
\Gamma_{a_K}(m_*(\phi) v)(t) &=& \int_G m_*(\phi)(s,s^{-1}t) v(s, s^{-1}t) ds\\
 &=& \phi(t) \int_{G} v(s, s^{-1}t) ds = ( \phi \Gamma_{a_K}(v))(t).
\end{eqnarray*}
Hence
\[
\| \phi \Gamma_{a_K}(v) - \Gamma_{a_K}(v_n)\|_{A} \leq \|\Gamma_{a_K}\| \| m_*(\phi)v- v_n\|_{\hh}
\]
and therefore $\| \phi \Gamma_{a_K}(v) - \Gamma_{a_K}(v_n)\|_{A}\rightarrow_{n\to \infty} 0$.

By Lemma \ref{l_Gammaj},
$\Gamma_{a_K}(v_n)\in J_{A(G)}(E)$ for each $n\in \bb{N}$.
Fix $T \in J_{A(G)}(E)^\perp$. Then
\begin{equation}\label{eq_GaK}
\langle \phi \cdot T ,  \Gamma_{a_K}(v) \rangle = \langle T, \phi \Gamma_{a_K}(v)\rangle =
\lim_{n\to\infty} \langle T, \Gamma_{a_K}(v_n)\rangle = 0.
\end{equation}
Note that, by Lemma \ref{l_Gamma}, $\Gamma_{a_K}(v)$ is a compactly supported function from $A(G)$.

Let $(U_\alpha)_\alpha$ be a neighbourhood basis at $e$ consisting of
relatively compact neighbourhoods uniformly contained in a compact neighbourhood
of $e$ and  ordered inversely by the inclusion relation.
For each $\alpha$, let $b_\alpha$ be an element in $A(G) \cap C_c(G)$ so that
$\supp(b_\alpha) \subseteq U_\alpha$ and  $b_\alpha(e) =1$.
Set $e_\alpha = b_\alpha/ \|b_\alpha\|_1$. It is easy to check that $(e_\alpha)_\alpha$
is  a bounded approximate identity
of $L^1(G)$ in $A(G) \cap C_c(G)$,
all of whose elements are supported in a fixed compact neighbourhood of the identity.

Fix a compactly supported element $b$ of $A(G)$.
Then $(e_\alpha \otimes b)_{\alpha}$ $\subseteq$ $\ahg$ $\cap$ $C_c(G \times G)$, and we may assume that,
for a certain compact set $F \subseteq G$,
$\supp(e_\alpha \otimes b) \subseteq F \times F$,  for all $\alpha$.
By Lemma \ref{l_Gamma},
if $a_F\in A(G)$ is a compactly supported function taking value $1$ on $F$, then
$\Gamma_{a_F}(e_\alpha \otimes b) = e_{\alpha}\ast b \in A(G)$.
Moreover, by the continuity of the map $r\to L_r(b)$ from $G$ into $A(G)$, we have that
\[
f \ast b = \int_G f(r) L_r(b) \  dr, \ \ \ f \in L^1(G),
\]
where the integral is understood in the Bochner sense and is $A(G)$-valued.
We thus have
\begin{equation}\label{eq:using-bai-on-w}
\|\Gamma_{a_F}(e_\alpha\otimes b) - b\|_{A} = \| e_\alpha * b  - b\|_{A}  \leq \sup_{r \in U_\alpha} \|L_r(b) - b\|_A \rightarrow_\alpha \ 0.
\end{equation}
Now (\ref{eq_GaK}) implies that
$\langle \phi \cdot T ,  b \rangle = 0$, for all $b\in C_c(G) \cap A(G)$. Since the
compactly supported functions in $A(G)$ form a dense subset of $A(G)$, we
conclude that $\phi\cdot T = 0$.
If $A(G)$ has an approximate identity, say $(a_\alpha)_\alpha$, then
$$\langle T,\phi\rangle=\lim_\alpha\langle T,\phi a_\alpha\rangle=\lim_\alpha\langle \phi\cdot T,a_\alpha\rangle=0$$
and hence $E$ is a set of spectral synthesis.

Assume now that $\phi$ has compact support and
let $a \in A(G)\cap C_c(G)$ so that $a|_{\supp(\phi)}\equiv 1$; hence, $a\phi=\phi$.
Therefore,
\[
  \langle T, \phi \rangle=  \langle T, a\phi \rangle  = \langle  \phi \cdot T, a\rangle =0.
\]
Since this holds for  an arbitrary element $\phi$ of $I^c_{A(G)}(E)$,
we conclude that $T \in I_{A(G)}^c(E)^\perp$,
and thus $E$ is a set of local spectral synthesis for $A(G)$.
\end{proof}

Recall that a locally compact group $G$ is called a \emph{Moore group} \cite{moore}
if  each continuous irreducible unitary representation of $G$ is finite dimensional.
We refer the reader to \cite{palmer} for background on this class of groups. In the
sequel, we will use the fact that every Moore group is amenable and unimodular \cite[p.~1486]{palmer}.
Note that if $G$ is either virtually abelian 
(that is, contains an open abelian subgroup of finite index) or compact then $G$ is a Moore group
(see \cite[Theorems 1 and 2]{moore}).

Our next aim is Theorem \ref{t:Moore-groups}, the general structure of whose proof
is inspired by that of the proof of \cite[Theorem 4.11]{lt}.
We proceed with some preliminary facts.

Let $K$ be a compact subset of $G$.
In what follows we will view the space $C(K)$ of all continuous functions on $K$
as a subspace of bounded Borel functions on $G$, equipped with the uniform norm;
thus,  the elements of $C(K)$
will be considered as functions $f$ defined on the whole of $G$, and such that $f(s) = 0$ whenever $s\not\in K$.
In the sequel, we will need to make a distinction between the essential supremum norm and the supremum norm; 
thus, we write $\|\cdot\|_{\infty}$ for the former and $\|\cdot\|_{\sup}$ for the latter. 
Let $(f_i)_{i\in \bb{N}},(g_i)_{i\in \bb{N}} \subseteq  C(K)$ be  sequences
with
$\sum_{i=1}^{\infty} \|f_i\|_{\sup}^2 <\infty$ and
$\sum_{i=1}^{\infty}  \|g_i\|_{\sup}^2 <\infty$, and let
$w = \sum_{i=1}^{\infty} f_i \otimes g_i$ be the corresponding element of $C(K)\otimes_{\gamma} C(K)$.
For every $s\in G$, we have
$$\sum_{i=1}^{\infty} |f_i(s)|^2 \leq \sum_{i=1}^{\infty} \|f_i\|_{\sup}^2 < \infty;$$
similarly, $\sum_{i=1}^{\infty} |g_i(t)|^2 < \infty$, $t\in G$.
By the Cauchy-Schwarz inequality, the series
$\sum_{i=1}^{\infty} f_i(s) g_i(t)$ is absolutely convergent for all $s,t\in G$.
One can moreover verify that its sum does not depend on the particular representation of the element
$w\in C(K)\otimes_{\gamma} C(K)$.
We thus view the elements of $C(K)\otimes_{\gamma} C(K)$ as (bounded
measurable) functions on $G\times G$.

\begin{lemma}\label{l_welldf}
Let $G$ be a Moore group,
$K\subseteq G$ be a compact set
and $w\in C(K) \otimes_\gamma C(K)$. 
Then $w_{s,t}\in L^1(G)$ and $\|w_{s,t}\|_1\leq |K| \|w\|_{\gamma}$, for all $s,t\in G$.
\end{lemma}
\begin{proof}
Let
$(f_i)_{i\in \bb{N}},(g_i)_{i\in \bb{N}} \subseteq  C(K)$ be  sequences
with
$\sum_{i=1}^{\infty} \|f_i\|_{\sup}^2 <\infty$ and
$\sum_{i=1}^{\infty}  \|g_i\|_{\sup}^2 <\infty$
such that
$w = \sum_{i=1}^{\infty} f_i \otimes g_i$ in $C(K)\otimes_{\gamma} C(K)$.
We have
\begin{eqnarray}
\int_G |w(sr,r^{-1}t)| dr
&\leq & \sum_{i=1}^{\infty} \int_G |f_i(sr)| |g_i(r^{-1}t)| dr\nonumber\\
&\leq & \left(\sum_{i=1}^{\infty} \int_{s^{-1}K}  |f_i(sr)|^2 dr\right)^{\frac{1}{2}}
\left(\sum_{i=1}^{\infty} \int_{tK^{-1}}  |g_i(r^{-1}t)|^2 dr\right)^{\frac{1}{2}} \nonumber\\
&\leq &  |K| \left(\sum_{i=1}^{\infty} \|f_i\|_{\sup}^2\right)^{1/2}
\left(\sum_{i=1}^{\infty} \|g_i\|_{\sup}^2\right)^{1/2}.
\nonumber
\end{eqnarray}
The claim now follows.
\end{proof}

Let $G$ be a Moore group.
Fix a compact set $K\subseteq G$ and a function $w\in C(K) \otimes_{\gamma} C(K)$.
By Lemma \ref{l_welldf},
for every function $h\in L^{\infty}(G)$ and any $s,t\in G$,
the function $h w_{s,t}$ is integrable; we set
$$(h\circ w) (s,t) := \int_G h(r) w_{s,t}(r) dr.$$
Note that, by Lemma \ref{l_welldf},
\begin{equation}\label{eq_wst2}
|(h\circ w) (s,t)|\leq |K| \|h\|_{\infty} \|w\|_{\gamma}, \ \ \ s,t\in G.
\end{equation}

As customary, by $\widehat{G}$ we denote the set of all (equivalence classes of)
continuous irreducible unitary representations of
the group $G$.
For $\pi \in \widehat{G}$,
let $H_{\pi}$ be the Hilbert space on which $\pi$ acts.
Setting $d_{\pi} = \dim H_{\pi}$,
let $\{e_{i}^{\pi}\}_{i = 1}^{d_{\pi}}$ be an orthonormal basis of $H_{\pi}$.
Denote by $\pi_{i,j}$ the corresponding coefficient functions of $\pi$, that is,
the functions given by $\pi_{i,j}(s) = (\pi(s) e_i^{\pi}, e_j^{\pi})$, $s\in G$.
By Lemma \ref{l_welldf}, the integral
\begin{equation}\label{eq:w-pi}
w^\pi(s,t) := \int_G w_{s,t}(r) \pi(r) dr
\end{equation}
is well-defined as an element of $\cl B(H_{\pi})$ for all $s,t\in G$.
Let
$$\tilde{w}^\pi (s,t) := \pi(s) w^\pi(s,t), \ \ \ s,t\in G.$$
For all $i,j = 1,\dots,d_{\pi}$, set
\[
w_{i,j}^\pi(s,t) := (w^\pi(s,t)e^{\pi}_i,e_j^{\pi})
\]
and
\[
\tilde{w}^\pi_{i,j} (s,t) := ( \tilde{w}^\pi(s,t) e_i^{\pi}, e_j^{\pi} );
\]
note that
$$w_{i,j}^\pi(s,t) = (\pi_{i,j}\circ w)(s,t).$$

\begin{lemma}\label{l_incbmh}
Let $G$ be a Moore group, $K\subseteq G$ be a compact set,
$\pi\in \widehat{G}$ and $w \in C(K)\otimes_{\gamma} C(K)$.
Then the functions $w_{i,j}^\pi$ and $\tilde{w}^\pi_{i,j}$ belong to $M^{\cb}\ahg$ for all $i,j = 1,\dots,d_{\pi}$.
Moreover, $\tilde{w}^\pi_{i,j}$ lies in the range of the map $m_*$ from Theorem \ref{th_daws}.
\end{lemma}
\begin{proof}
Fix $i,j\in \{1,\dots,d_{\pi}\}$ and let $K\subseteq G$ be a compact set.
Define $\Lambda : C(K)\times C(K)\to A(G)$ by
$\Lambda(f,g) = (f\pi_{i,j})\ast g$.
Then
$$\|\Lambda(f,g)\|_{A} \leq \|f\pi_{i,j}\|_2 \|g\|_2 \leq |K| \|f\|_{\sup}\|g\|_{\sup};$$
in other words, $\Lambda$ is a bounded (bilinear) map and hence induces a map
(denoted in the same fashion) $\Lambda : C(K)\otimes_{\gamma} C(K)\to A(G)$ so that
$$\|\Lambda(w)\|_{A}\leq |K| \|w\|_{\gamma}, \ \ \ w\in C(K)\otimes_{\gamma} C(K).$$
Let $\Psi = m_* \circ \Lambda$; by Theorem \ref{th_daws}, the map
$$\Psi : C(K)\otimes_{\gamma} C(K)\to M^{\cb}\ahg$$
is bounded with $\|\Psi\|\leq |K|$.

Let $\Phi : C(K)\otimes_{\gamma} C(K)\to \ell^{\infty}(G\times G)$ be given by
$\Phi(w) = \tilde{w}_{i,j}^\pi$.
By the definition of $\tilde{w}_{i,j}^\pi$ and Lemma \ref{l_welldf},
$\|\Phi(w)\|_{\infty} \leq |K| \|w\|_{\gamma}$.
If $f,g\in C(K)$, $w = f\otimes g$ and $s,t\in G$ then
\begin{eqnarray*}
\Phi(w)(s,t) & = &
\int_G w(sr,r^{-1}t) (\pi(sr)e_i^{\pi}, e_j^{\pi}) dr
= \int_G w(r,r^{-1}st) \pi_{i,j}(r) dr\\
& = & \int_G f(r) \pi_{i,j}(r) g(r^{-1}st) dr = m_*(f\pi_{i,j}*  {g})(s,t) = \Psi(w)(s,t).
\end{eqnarray*}
By linearity,
$$\Phi(w)(s,t)  = \Psi(w)(s,t), \ \ \ s,t\in G, \  w\in C(K)\odot C(K).$$
Let $w\in C(K)\otimes_{\gamma} C(K)$ and $(w_k)_{k\in \bb{N}}\subseteq C(K)\odot C(K)$
be a sequence with $\|w_k - w\|_{\gamma}\to_{k\to \infty} 0$. By (\ref{eq_unif}),
$$\Psi(w_k)(s,t)\to \Psi(w)(s,t), \ \ \ s,t\in G.$$
On the other hand, since $\Phi$ is bounded,
$$\Phi(w_k)(s,t)\to \Phi(w)(s,t), \ \ \ s,t\in G.$$
It follows that $\Phi(w)(s,t) = \Psi(w)(s,t)$ for all $s,t\in G$;
since $\Psi(w)\in M^{\cb}\ahg$, we conclude that $\tilde{w}_{i,j}^\pi\in M^{\cb}\ahg$.

A calculation similar to the one in \cite[(4.6)]{lt} implies that
\begin{equation}\label{eq:Moore-needed}
w_{i,j}^\pi(s,t) = \sum_{l=1}^{d_{\pi}} \pi_{l,j}(s) \tilde{w}^\pi_{i,l} (s,t).
\end{equation}
Since $\pi_{l,j}\in B(G)$, Proposition \ref{p_algtenin} implies that
$\pi_{l,j}\otimes 1\in M^{\rm cb}\ahg$.
By (\ref{eq:Moore-needed}), $w_{i,j}^\pi \in M^{\rm cb}\ahg$.
\end{proof}

Let $w \in \ahg$. By Lemma \ref{l_newinde} (ii), the function from $G$ into $\ahg$, mapping
$r$ to $w_r$, is continuous.
Therefore, if $f\in L^1(G)$ then
\begin{equation}\label{eq:f.w}
f \star w := \int_G f(r) w_r dr
\end{equation}
is a well-defined $\ahg$-valued integral in Bochner's sense.

\begin{lemma}\label{l_star}
Let $G$ be a Moore group and $w \in \ahg$.
If $(e_\alpha)_{\alpha}$ is a bounded approximate identity of  $L^1(G)$ then
$e_{\alpha}\star w\to_{\alpha} w$.
\end{lemma}
\begin{proof}
Let $(U_{\alpha})_{\alpha}$ be a basis of neighbourhood of $e$,
directed by inverse inclusion, and
$(f_\alpha)_{\alpha}\subseteq L^1(G)$ be such that $\supp(f_{\alpha}) \subseteq U_{\alpha}$
and $\|f_\alpha\|_1=1$ for all $\alpha$.
As in the proof of Proposition~\ref{p_ccehag},
if $\cl X =A(G) \cap C_c(G)$ then $\cl X \odot \cl X$ is dense in $\ahg$;
in addition, $\cl X \odot \cl X \subseteq A(G \times G)$.
For a given $\epsilon > 0$, choose $v \in\cl X \odot \cl X$ so that
$\| w - v \|_{\rm h} < \epsilon/3$. By Lemma \ref{l_newinde} (i),
\[
\|w_r   - v_r\|_{\rm h} < \epsilon/3, \ \ \  r \in G.
\]
Since $v\in A(G\times G)$, there is a neighbourhood $V$ of $e$ such that
\[
\|v_r -v \|_{A} < \frac{\epsilon}{3}, \ \ \ r\in V.
\]
Let $\alpha_0$ be such that $U_{\alpha_0} \subseteq V$.
For all $\alpha\geq \alpha_0$ we have
\begin{eqnarray*}
\|f_{\alpha}\star w - w\|_{\hh}
& = & \left\|  \int_G f_\alpha(r) w_r   dr - w \right\|_{\rm h}
\leq   \int_V \ |f_\alpha(r)|  \|w_r   -  w  \|_{\rm h} dr\\
&\leq &  \int_V |f_\alpha(r)| \left( \|w_r  -  v_r \|_{\rm h}  +  \|v_r - v\|_{\rm h} + \|v   -  w\|_{\rm h} \right) dr\\
&\leq & 2 \| w   -   v  \|_{\rm h}  +   \int_V |f_\alpha(r)|\   \|v_r  - v   \|_{A}  dr < \epsilon.
\end{eqnarray*}
Thus, $f_{\alpha}\star w \to_{\alpha} w$.
It is immediate that $(f_\alpha)$ is a bounded approximate identity of $L^1(G)$.
By Cohen's factorisation theorem \cite[32.22]{hr2} and  \cite[32.33(a)]{hr2},
$e_\alpha \star w  \rightarrow w$ for any bounded approximate identity $(e_\alpha)_\alpha$ in $L^1(G)$.
\end{proof}

\begin{remark}\label{r_eqtwa}
If $K\subseteq G$ is a compact set, $w\in C(K)\otimes_{\gamma} C(K)$ and
$f\in L^{\infty}(G)$ is compactly supported then $f\circ w = f\star w$.
\end{remark}
\begin{proof}
For $s,t\in G$ we have
\begin{eqnarray*}
(f \star w)(s,t)
& = & \langle f \star w,\lambda_s\otimes\lambda_t\rangle
= \left\langle \int_G f(r) w_r dr, \lambda_s\otimes\lambda_t\right\rangle\\
& = & \int_G f(r) \langle w_r, \lambda_s\otimes\lambda_t\rangle dr
= \int_G f(r) w_r(s,t) dr\\
& = & \int_G f(r) w_{s,t}(r) dr = (f\circ w)(s,t).
\end{eqnarray*}
\end{proof}

\begin{theorem}\label{t:Moore-groups}
Let $G$ be a   Moore group and $E \subseteq G$ be a closed set.
If $E$ is a set of spectral synthesis for $A(G)$
then $E^\sharp$ is a set of spectral synthesis for $\ahg$.
\end{theorem}
\begin{proof}
Let $E\subseteq G$ be a set of spectral synthesis for $A(G)$
and $w\in I_{\ahg}(E^\sharp)$.
Since $G$ is amenable, Theorem \ref{p_HWA} implies that
$w$ can be approximated by compactly supported functions in $I_{\ahg}(E^\sharp)$;
we may thus assume that $w$ is compactly supported itself.
Let $K\subseteq G$ be a compact set such that $\supp(w)\subseteq K\times K$.
We have that $w\in C_0(G)\otimes_{\hh} C_0(G)$;
by Grothendieck's inequality,
$w\in C_0(G) \otimes_\gamma C_0(G)$.
Write $w = \sum_{k=1}^{\infty} f_k\otimes g_k$, where $(f_k)_{k\in \bb{N}}$ and $(g_k)_{k\in \bb{N}}$
are families of functions in $C_0(G)$ such that
$\sum_{k=1}^{\infty} \|f_k\|_{\infty}^2  < \infty$ and $\sum_{k=1}^{\infty} \|g_k\|_{\infty}^2 < \infty$.
Let $\tilde{f}_k = f_k\chi_{K}$ (resp. $\tilde{g}_k = f_k\chi_{K}$).
Then $\tilde{f}_k, \tilde{g}_k\in C(K)$ for each $k\in \bb{N}$;
moreover,
$\sum_{k=1}^{\infty} \|\tilde{f}_k\|_{\sup}^2  < \infty$ and $\sum_{k=1}^{\infty} \|\tilde{g}_k\|_{\sup}^2 < \infty$.
Letting $\tilde{w}_m = \sum_{k=1}^{m} \tilde{f}_k\otimes \tilde{g}_k$, $m\in \bb{N}$,
we thus have that the sequence $(\tilde{w}_m)_{m\in \bb{N}}$ converges
in $C(K)\otimes_{\gamma} C(K)$; let $\tilde{w}$ be its limit.
Since the uniform norm is dominated by the projective one,
we easily see that $w(s,t) = \tilde{w}(s,t)$ for all $s,t\in K$.
Thus, $w\in C(K)\otimes_{\gamma} C(K)$.

Since $w \in I_{\ahg}(E^\sharp)$, we have that $w_r|_{E^{\sharp}} = 0$ for all $r\in G$,
and hence the functions $w^\pi_{i,j}$ and $\tilde{w}^\pi_{i,j}$ vanish on
$E^\sharp$ for all $\pi \in \widehat{G}$ and all $1\leq i,j \leq d_\pi$.
In the sequel, we fix
$u \in \ev(G)$ with  $\supp(u) \subseteq E^\sharp$.
We divide the rest of the proof in three steps.

\medskip

{\it Step 1.}   $w^\pi_{i,j} \cdot u = 0$ for all $\pi\in \widehat{G}$ and all $i,j = 1,\dots,d_{\pi}$.

\medskip

Fix $\pi \in \widehat{G}$ and $i,j \in \{1,\dots,d_\pi\}$.
By Lemma \ref{l_incbmh}, there exists $a \in  A(G)$ such that $m_*(a) = \tilde{w}^\pi_{i,j}$.
Since $\tilde{w}^\pi_{i,j}$ vanishes on $E^\sharp$,
we have that $a \in I_{A(G)}(E)$.
Since $E$ is a set of spectral synthesis for $A(G)$, there exists a sequence
$(a_n)_{n\in \bb{N}}\subseteq A(G)$, whose elements have compact supports disjoint from $E$,
such that $\|a_n - a\|_A \rightarrow_{n\to\infty} 0$.
Note that the element $m_*(a_n)$ of $M^{\rm cb}\ahg$ vanishes on a neighbourhood of
$E^\sharp$ for each $n\in \bb{N}$.
By Theorem \ref{th_daws}, if $w'' \in \ahg$ then $m_*(a_n) w''\in \ahg$. Moreover, if $w''$ is compactly supported then
$m_*(a_n) w''$ is compactly supported and vanishes on a neighbourhood of $E^\sharp$;
hence, $m_*(a_n) w'' \in J_{\ahg}(E^{\sharp})$.
By Proposition \ref{p_ccehag}, every element $w'$ of $\ahg$ is the limit of
compactly supported elements of $\ahg$. It follows that $m_*(a_n) w' \in J_{\ahg}(E^{\sharp})$
for every $w'\in \ahg$ and every $n\in \bb{N}$. Therefore,
\[
\langle\tilde{w}^\pi_{i,j} \cdot u, w' \rangle
= \langle u, \tilde{w}^\pi_{i,j}  w' \rangle = \langle u, m_*(a) w'\rangle =
\lim_{n\to\infty} \langle u, m_*(a_n) w'\rangle = 0,
\]
for every $w'\in \ahg$.
This shows that $\tilde{w}^\pi_{i,j} \cdot u = 0$;
by (\ref{eq:Moore-needed}),
$w^\pi_{i,j} \cdot u = 0.$

\medskip

{\it Step 2.}
If $\supp_{\hh}(u) \subseteq K\times K$ then
$(f \star w)\cdot u  = 0$ for all  $f\in L^1(G)$.

\medskip

Let $U$ be an open relatively compact subset of $G$ such that $K \subseteq U$. 
Let $a \in A(G) \cap C_c(G)$ so that $a|_{\overline{U}}\equiv 1$. Let $F\subseteq G$ be a compact set
such that $F^{-1}=F$ and $\supp(a)\subseteq F$.
For a compactly supported element $f \in L^1(G)$,
using Lemma \ref{l_supp-equality}, for all $v \in \ahg$ we have
\begin{eqnarray}\label{eq:on}
\langle u, (f \star w) v  \rangle
& = &
\langle (a \otimes a) \cdot u, (f \star w) v\rangle
= \langle u, (a\otimes a) ( f \star w) v \rangle  \nonumber \\
& = & \langle u, \chi_{F\times F} (a\otimes a)  (f \star w) v\rangle.
\end{eqnarray}  
On the other hand, using Remark \ref{r_eqtwa} we have
\begin{eqnarray*}
\chi_{F \times F} (f \star w)(s,t)
& = & \int_G \chi_F(s) \chi_F(t) f(r) w(sr,r^{-1}t) dr\\
&=& \chi_F(s) \chi_F(t) \int_{F^{-1}F}  f(r) w( sr,r^{-1}t) dr\\
&=&  \chi_{F \times F} ( (\chi_{F^{-1}F}f) \star w) (s,t).
\end{eqnarray*}
Now (\ref{eq:on}) implies
\begin{eqnarray*}
\langle u, (f \star w) v \rangle &=& \langle u, (a \otimes a) \chi_{F \times F} ((\chi_{F^{-1}F} f) \star w) v \rangle \\
&=& \langle u, (a \otimes a)((\chi_{F^{-1}F}f)\star w)   v\rangle \\
&=&  \langle (a\otimes a) \cdot u, (  (f \chi_{F^{-1}F}) \star w)  v\rangle\\
 &=& \langle u, ( (f \chi_{F^{-1}F}) \star w)   v\rangle.
\end{eqnarray*}
Thus,
\begin{equation}\label{eq:f-star-w}
(f \star w) \cdot u = (f \chi_{F^{-1}F} \star w) \cdot u.
\end{equation}
If $\pi \in \widehat{G}$, $i,j \in \{1, \ldots, d_\pi\}$ and $v\in \ahg$ then
\begin{eqnarray*}
(a\otimes a) w_{i,j}^\pi (s,t)
& = &
a(s) a(t) \int_G \pi_{i,j}(r) w_{s,t}(r)dr\\
& = &
a(s) a(t) \chi_F(s) \chi_F(t) \int_G \pi_{i,j}(r) w(sr,r^{-1}t)dr\\
& = &
a(s) a(t) \int_G \pi_{i,j}(r)\chi_{F^{-1} F}(r) w(sr,r^{-1}t)dr\\
&=& (a\otimes a) (\pi_{i,j}\chi_{F^{-1}F} \circ w)(s,t).
\end{eqnarray*}
By Remark~\ref{r_eqtwa} and Step~1 we now have
 \begin{eqnarray*}
((\pi_{i,j} \chi_{F^{-1}F})\star w) \cdot u
& = &
(\pi_{i,j}\chi_{F^{-1}F} \circ w) \cdot u\\
& = & (\pi_{i,j}\chi_{F^{-1}F} \circ w) \cdot ((a\otimes a)\cdot u)\\
& = &
(a\otimes a)(\pi_{i,j}\chi_{F^{-1}F} \circ w) \cdot u\\
& = &
(a\otimes a)w_{i,j}^\pi \cdot u
=
w_{i,j}^\pi \cdot ((a\otimes a) \cdot u))
= w_{i,j}^\pi \cdot u = 0.
\end{eqnarray*}
Since $G$ is unimodular, by \cite[13.6.5]{dix}, 
$f \chi_{F^{-1}F}$ can be approximated  in $L^1(G)$ by finite linear combinations of the form
\begin{equation}\label{eq:form-of-bai}
\sum_{k=1}^m  c_k \chi_{F^{-1}F} f_k
\end{equation}
where, for every $k$, the function
$f_k$ has the form $\pi_{i,j}$ for some $\pi \in \widehat{G}$ and some
$i,j\in \{1,\dots,d_{\pi}\}$.
Hence, by \eqref{eq:f-star-w},
 \[
 (f \star w) \cdot u= (f \chi_{F^{-1}F} \star w) \cdot u = 0.
 \]

\medskip

{\it Step 3.} $\langle u,  w \rangle  = 0$.

\medskip

Let  $(\varepsilon_\beta)_{\beta}$ be  a bounded approximate identity of $A(G)$,
such that $\supp(\varepsilon_\beta)$ $\subseteq K_\beta$ for some compact set $K_{\beta}\subseteq G$.
We can assume, without loss of generality,
that $K\subseteq K_{\beta}$ for each $\beta$.
Assume first that $u$ is compactly supported and let $(e_{\alpha})_{\alpha}$ be a bounded approximate 
identity for $L^1(G)$. Using Step 2 and Lemma \ref{l_star}, we have 
 \begin{eqnarray*}
 \langle u, w\rangle &=& \lim_\beta \langle u, (\varepsilon_\beta \otimes \varepsilon_\beta) w\rangle \\
&=&  \lim_\beta \langle (\varepsilon_\beta\otimes \varepsilon_\beta) \cdot u, w\rangle \\
 &=& \lim_\beta \lim_\alpha \langle(e_\alpha \star w) \cdot u, \varepsilon_\beta\otimes \varepsilon_\beta\rangle =0.
 \end{eqnarray*}
 If $u$ is arbitrary then
 \[
 \langle u,w\rangle = \lim_\beta \langle u, (\varepsilon_\beta \otimes \varepsilon_\beta) w\rangle 
 = \lim_\beta \langle (\varepsilon_\beta \otimes \varepsilon_\beta) \cdot u, w\rangle =0.
 \] 
 
We have thus shown that
$\langle u,w\rangle = 0$ whenever
$w\in I_{\ahg}(E^\sharp)$ and $u\in \ev(G)$ is supported in $E^\sharp$.
This shows that $E^{\sharp}$ is a set of spectral synthesis for $\ahg$.
\end{proof}

\noindent {\bf Remark.}
We note that, in the proof of Theorem~\ref{t:Moore-groups},
the fact that $G$ is a Moore group was essentially used
in the finiteness of the sum (\ref{eq:Moore-needed}).

\begin{corollary}\label{c:moore-antidiagonal}
Let $G$ be a Moore group. A closed set $E\subseteq G$ is a set of spectral synthesis for $A(G)$
if and only if $E^{\sharp}$ is a set of spectral synthesis for $\ahg$.
\end{corollary}
\begin{proof}
Immediate from Theorems \ref{t_spectral-of-E-sharp} and \ref{t:Moore-groups}.
\end{proof}

The subset
$$\tilde{\Delta} = \{(s,s^{-1}): s\in G\}$$
of $G\times G$ is usually referred to as the
\emph{antidiagonal} of $G$.
It is known that if the group $G$ is compact,
the antidiagonal is not a set of spectral synthesis for $A(G \times G)$
unless the connected component of the neutral element is abelian (see \cite[Theorem~2.5]{fss}).
On the other hand,
the antidiagonal coincides with $\{e\}^{\sharp}$;
since $\{e\}$ is a set of spectral synthesis for $A(G)$,
Theorem \ref{t:Moore-groups} implies that
$\tilde{\Delta}$ is a set of spectral synthesis for $\ahg$ if $G$ is a Moore group.
In Section \ref{s_ibm}, we will refine this statement and give a characterisation of all elements
in the dual of $\ahg$ supported in the antidiagonal for more general groups.


\section{The case of virtually abelian groups}\label{s_vag}

It is easy to see that the flip of variables preserves spectral synthesis in the algebra $A(G\times G)$. 
The question of whether the same holds true for the algebra $\ahg$ is the motivation behind the present section. 
Recall that a locally compact group is called virtually abelian, if it has an open abelian subgroup of finite index. 
We assume, in this section, that $G$ is a virtually abelian group. 
We first give a general result on the extended Haagerup tensor product; in the case of the Haagerup tensor product,
it was established in \cite{It}.

\begin{proposition}\label{p_flipeh}
Let $\cl M$ be a unital C*-algebra. The following are equivalent:

(i) \ $\cl M$ is subhomogeneous;

(ii) the linear map $\frak{f}: \cl M\odot \cl M \to \cl M\otimes_{\eh} \cl M$, given on elementary tensors by $\frak{f}(a\otimes b) = b\otimes a$,
extends to a completely bounded map on $\cl M\otimes_{\eh} \cl M$. 
\end{proposition}
\begin{proof}
(i)$\Rightarrow$(ii)
Suppose that $\cl M \subseteq \cl Z \otimes \mathbb{M}_n(\Bbb{C})$ for some $n$ and some commutative
von Neumann algebra $\cl Z$. 
Assume that $\cl Z$ coincides with the multiplication masa of $L^{\infty}(X,\mu)$, 
acting on the Hilbert space $L^2(X,\mu)$, for some suitably chosen measure space $(X,\mu)$. 
Denote by $\gamma$ the involution on $\cl Z$, that is, $\gamma(f) = f^*$, $f\in \cl Z$. 
Let $\tau_n$ be the matrix transpose acting on $\mathbb{M}_n(\Bbb{C})$.

If $a\in \cl M$ then 
$a^*=(\gamma \otimes\tau_n)(a)$. 
We note first that $\gamma \otimes\tau_n$ is completely bounded.
Indeed, if $m\in \bb{N}$ then making the identification 
$\mathbb{M}_m(\cl Z\otimes \mathbb{M}_n) \equiv \cl Z\otimes \mathbb{M}_m(\mathbb{M}_n)$, we have that 
the map $(\gamma\otimes\tau_n)^{(m)}$ corresponds to $\gamma\otimes\tau_n^{(m)}$.
Since $\tau_n$ is completely bounded with $\|\tau_n\|_{\rm cb} = \|\tau_n^{(n)}\| = n$, 
we have that $\|(\gamma \otimes\tau_n)^{(m)}\|\leq n$ for every $m$.

Let $x \in \cl M\otimes_{\eh}\cl M$ and write $x = A\odot B$, where 
$A$ is the row operator $(a_{\alpha})_{\alpha\in \bb{A}}$, while $B$ is the column operator $(b_\alpha)_{\alpha\in \bb{A}}$. 
We have 
\begin{eqnarray*}
\left\|\sum_{\alpha\in \bb{A}} a_\alpha^* a_\alpha\right\| 
& = & 
\left\|(\gamma \otimes\tau_n^{(\infty)})(A) (\gamma \otimes\tau_n^{(\infty)})(A^*)\right\|
\leq 
\left\|\gamma \otimes\tau_n\right\|_{\rm cb}^2 \|A\|^2\\
& = & \left\|\gamma \otimes\tau_n\right\|_{\rm cb}^2 \|AA^*\|
\leq n^2 \left\|\sum_{\alpha\in \bb{A}} a_\alpha a_\alpha^*\right\|.
\end{eqnarray*}
A similar argument shows that
$$ \left\|\sum_{\alpha\in \bb{A}} b_\alpha b_\alpha^*\right\| \leq n^2 \left\|\sum_{\alpha\in \bb{A}} b_\alpha^* b_\alpha\right\|,$$
and (ii) is established.

(ii)$\Rightarrow$(i) 
We have that $\cl M\otimes_{\rm h}\cl M \subseteq \cl M\otimes_{\eh}\cl M$ completely isometrically, 
and that the algebraic tensor product $\cl M\odot \cl M$ is dense in $\cl M\otimes_{\rm h}\cl M$.
It follows that the map $\frak{f}$ leaves $\cl M\otimes_{\rm h}\cl M$ invariant. 
By \cite[Theorem 4]{It}, $\cl M$ is subhomogeneous. 
\end{proof}

\begin{theorem}\label{th_vaf}
Let $G$ be a locally compact group. The following are equivalent:

(i) \ $G$ is virtually abelian;

(ii) the linear 
map $\sigma : A(G)\odot A(G) \to A(G) \otimes_{\rm h} A(G)$, given on elementary tensors by $\sigma(\phi \otimes \psi) = \psi \otimes \phi$,
extends to a completely bounded map on $\ahg$.
\end{theorem}
\begin{proof}
(i)$\Rightarrow$(ii) By \cite{moore}, the unitary representations of $G$ have uniformly bounded 
dimension, and hence $C_r^*(G)$ is subhomogeneous. This easily implies that 
$\vn(G)$ is subhomogeneous and, by Proposition \ref{p_flipeh},
the flip extends to a completely bounded map $\frak{f}$ on $\vn_{\eh}(G)$. 
Note that $\frak{f}$ is weak* continuous; indeed, suppose that $(u_i)_i$ is a bounded net that converges to an 
element $u\in \vn_{\eh}(G)$ in the weak* topology. For $\phi,\psi\in A(G)$ we then have
$$\langle \frak{f}(u_i),\phi\otimes\psi\rangle = \langle u_i,\psi\otimes\phi\rangle \to 
\langle u,\psi\otimes\phi\rangle = \langle \frak{f}(u),\phi\otimes\psi\rangle;$$
by the uniform boundedness of $(u_i)_i$ and the density of the algebraic tensor product $A(G)\odot A(G)$ in $\ahg$, 
we have that $\frak{f}(u_i)\to \frak{f}(u)$ in the weak* topology. It follows that the map $\frak{f}$ has a 
completely bounded predual; a straightforward argument shows that this predual coincides with $\sigma$ on 
$A(G)\odot A(G)$.

(ii)$\Rightarrow$(i) The dual of the map $\sigma$ is easily seen to coincide with the flip on the algebraic tensor product 
$\vn(G)\odot \vn(G)$. By Proposition \ref{p_flipeh}, $\vn(G)$ is subhomogeneous. By 
the proof of \cite[Proposition 1.5]{fr}, $G$ is virtually abelian. 
\end{proof}

\begin{corollary}\label{c_flipss}
Let $G$ be a virtually abelian locally compact group and let  $E$ be a closed subset of $G \times G$
that satisfies spectral synthesis in $A_{\rm h}(G)$. Then the set 
$\tilde E:=\{(s,t): (t,s) \in E\}$ satisfies spectral synthesis in $A_{\rm h}(G)$.
\end{corollary}
\begin{proof}
By Theorem \ref{th_vaf}, the flip $\sigma$ extends to a completely bounded map on $\ahg$. 
It is easy to see that, if $w\in \ahg$ then $\sigma(w)(s,t) = w(t,s)$, $s,t\in G$.
Thus, $\sigma$ carries $I_{\ahg}(E)$ (resp. $J_{\ahg}(E)$) onto $I_{\ahg}(\tilde{E})$
(resp. $J_{\ahg}(\tilde{E})$). The claim is now clear. 
\end{proof}

It is well-known that the linear map $S : A(G)\to A(G)$ given by $S(u)(s) = u(s^{-1})$,
is an isometry, and that it is completely bounded if and only if $G$ is
virtually abelian \cite[Proposition~1.5]{fr}.
The adjoint $S^* : \vn(G)\to \vn(G)$ of $S$ is given by
$S^*(\lambda_s)= \lambda_{s^{-1}}$; clearly, $S^*$ is weak* continuous.

Suppose that $G$ is a virtually abelian group.
Let $N : A(G) \rightarrow M\ahg$ be the map given by
\begin{equation}\label{eq:N}
N(a)(s,t) = a(st^{-1}), \ \ \ s,t\in G,
\end{equation}
that is, $N(a) = (\id \otimes S)\circ m_*(a)$.
If $v \in \ahg$ and $a\in A(G)$ then, by Theorem \ref{th_daws},
\[
\|N(a) v \|_{\hh} = \|(\id \otimes S)\left(m_* (a)(v)\right)\|_{\hh}
\leq   \|S\|_{\rm cb} \|a\|_{A} \|v\|_{\hh};
\]
thus, $N$ is bounded.
One can modify the proof of Theorem~\ref{t:Moore-groups}
where the function $w^{\pi}$ defined in (\ref{eq:w-pi}) is replaced by the function
\[
(s,t)\to \int_G w(sr,tr) \pi(r) dr,
\]
a similar change is implemented in (\ref{eq:f.w}), and
where the map $m_*$ is replaced by the map $N$.
The modified proof shows that if $E$ is  a set of spectral synthesis for $A(G)$
then  
$$E^{\flat} = \{(s,t) \in G \times G: \ st^{-1} \in E\}$$ 
is a set of spectral synthesis for $\ahg$.

Similarly, the proof of Theorem~\ref{t_spectral-of-E-sharp} can be modified by using the map $\hat{\Gamma}$
given by
\[
\hat{\Gamma}(v)(t) = \int_G v(ts,s) ds
\]
(note that $\hat{\Gamma}(a\otimes b) = S(b\ast S(a))$).
Thus instead of (\ref{eq:using-bai-on-w}), one can show that
$\|\hat{\Gamma}(w \otimes e_\alpha) - w\|_{A} \rightarrow_{\alpha} 0$.
We also have $\hat{\Gamma}(N(f) v) =   f\hat{\Gamma}(v)$.
Working with $N$ in the place of $m_*$, a modification of the proof of Theorem \ref{t_spectral-of-E-sharp}
shows that if $E^{\flat}$ is a set of spectral synthesis for $\ahg$ then
$E$ is a set of spectral synthesis for $A(G)$.

Let 
$$E^* = \{(s,t)\in G\times G : ts^{-1} \in E\}.$$
Combining the observations in the last two paragraphs with Corollary \ref{c_flipss},
we obtain the following corollary.

\begin{corollary}\label{c:generalizing-Nico}
Let $G$ be a virtually abelian group and $E \subseteq G$ be a closed set.
The following conditions are equivalent:

(i) \ \ $E$ is a set of spectral synthesis for $A(G)$;

(ii) \ $E^\sharp$ is a set of spectral synthesis for $\ahg$;

(iii) $E^*$ is a set of spectral synthesis for $\ahg$.
\end{corollary}

We summarise some implications of 
Corollary~\ref{c:generalizing-Nico} and  \cite[Proposition~3.1]{rs} in the next remark.

\begin{remark}\label{r_Nico-extension}
Let $G$ be a virtually abelian compact group. Then the following are equivalent:

(i) \ \ $E$ is a set of spectral synthesis for $A(G)$;

(ii) \ $E^\sharp$ is a set of spectral synthesis for $C(G) \otimes_\gamma C(G)$;

(iii) $E^*$ is a set of spectral synthesis for $A(G\times G)$;

(iv)\ $E^*$ is a set of spectral synthesis for $\ahg$;

(v)\ \ $E^\sharp$ is a set of spectral synthesis for $\ahg$.
\end{remark}


\section{$\vn(G)'$-bimodule maps and supports}\label{s_ibm}

This section is centred around the
correspondence between the elements of the extended Haagerup tensor product
$\ev(G) = \vn(G)\otimes_{\eh} \vn(G)$ and the completely bounded
weak* continuous $\vn(G)'$-module maps on $\cl B(L^2(G))$.
We assume throughout the section that $G$ is a second countable locally compact group,
and, in what follows,
relate the support of an element $u\in \ev(G)$ to certain
invariant subspaces of the completely bounded map corresponding to $u$
(see Theorem \ref{th_invars} and Corollary \ref{c_no1}).
Further, we characterise the elements $u\in \ev(G)$
supported on the antidiagonal as those, for which the corresponding
completely bounded map leaves the multiplication masa of $L^{\infty}(G)$ invariant
(Theorem \ref{th_antid}).
It is well-known 
(see \cite{neufang_thesis} and \cite{nrs}) that the latter class consists precisely of the maps of the form
$\Theta(\mu)$ with $\mu\in M(G)$, where
$$\Theta(\mu)(T) = \int_G\lambda_s T\lambda_s^* d\mu(s), \ \ \ T\in \cl B(L^2(G)).$$

Note that, since $\lambda(L^1(G))$ is a (weak* dense) subspace of $\vn(G)$,
the algebraic tensor product
$\lambda(L^1(G))\odot \lambda(L^1(G))$ sits naturally inside $\ev(G)$.
We refer the reader to (\ref{eq_Phi_u}) for the definition of the map $\Phi_u$
associated with an element $u$ of $\ev(G)$.

\begin{lemma}\label{l_appr}
Let $u\in \ev(G)$. Then there exists a net $(u_{\alpha})_{\alpha}\subseteq \lambda(L^1(G))\odot \lambda(L^1(G))$
such that

(i) \ \ $\|u_{\alpha}\|_{\eh}\leq \|u\|_{\eh}$ for all $\alpha$,

(ii) \ $u_{\alpha}\to u$ in the weak* topology of $\ev(G)$, and

(iii) $\Phi_{u_{\alpha}}(x)\to_\alpha \Phi_u(x)$ in the weak* topology of $\cl B(L^2(G))$,
for every $x\in \cl B(L^2(G))$.
\end{lemma}
\begin{proof}
Suppose that $u = \sum_{i=1}^{\infty}  a_i\otimes b_i$ is a w*-representation of $u$
with the property that
\[
\|u\|_{\eh} = \left\|\sum_{i=1}^{\infty}  a_ia_i^*\right\|^{\frac{1}{2}} \left\|\sum_{i=1}^{\infty}  b_i^*b_i\right\|^{\frac{1}{2}}.
\]
Let $\cl F(\ell^2)$ be the algebra of all operators of finite rank on $\ell^2$.
Then the algebraic tensor product $\cl F(\ell^2)\odot \lambda(L^1(G))$ is a
weak* dense *-subalgebra of the von Neumann algebra $\cl B(\ell^2)\bar{\otimes} \vn(G)$.
Realise $\cl B(\ell^2)\bar{\otimes} \vn(G)$ as a space of matrices (of infinite size)
with entries in $\vn(G)$, and note that, since
$\sum_{i=1}^{\infty} a_i a_i^*$
(resp. $\sum_{i=1}^{\infty} b_i^* b_i$)
is weak* convergent, $(a_i)_{i=1}^{\infty}$ (resp. $(b_i)_{i=1}^{\infty}$) can be viewed as
an element $A$ (resp. $B$) of $\cl B(\ell^2)\bar{\otimes} \vn(G)$ supported by the first row
(resp. by the first column).
By the Kaplansky Density Theorem, there exist nets
$(\tilde{A}_{\alpha})_{\alpha}$ and $(\tilde{B}_{\alpha})_{\alpha}$ in $\cl F(\ell^2)\odot \lambda(L^1(G))$
such that $\|\tilde{A}_{\alpha}\|\leq \|A\|$, $\|\tilde{B}_{\alpha}\|\leq \|B\|$ for all $\alpha$, and
$$\tilde{A}_{\alpha}\to\mbox{}_{\alpha} \ A, \ \ \tilde{B}_{\alpha}\to\mbox{}_{\alpha} \ B$$
in the strong* topology.
Let $A_{\alpha}$ (resp. $B_{\alpha}$) be the compression of $\tilde{A}_{\alpha}$
(resp. $\tilde{B}_{\alpha}$) to the first row (resp. column).
Then $\|A_{\alpha}\|\leq \|A\|$, $\|B_{\alpha}\|\leq \|B\|$ for all $\alpha$, and
$$A_{\alpha}\to\mbox{}_{\alpha} \ A, \ \ B_{\alpha}\to\mbox{}_{\alpha} \ B$$
in the strong* topology. Let $u_{\alpha} = A_{\alpha}\odot B_{\alpha}$;
then $u_{\alpha}\in \lambda(L^1(G))\odot \lambda(L^1(G))$ and
\begin{equation}\label{eq_lessu}
\|u_{\alpha}\|_{\eh}\leq \|A_{\alpha}\|\|B_{\alpha}\| \leq \|A\|\|B\| = \|u\|_{\eh}
\end{equation}
for all $\alpha$.
If $x\in \cl B(L^2(G))$ then
$$\Phi_{u_{\alpha}}(x) = A_{\alpha}(1\otimes x) B_{\alpha} \to\mbox{}_{\alpha} \
A(1\otimes x)B$$ in the weak operator topology.
By (\ref{eq_lessu}),
$$\|\Phi_{u_{\alpha}}(x)\|\leq \|u_{\alpha}\|_{\eh}\|x\|\leq \|u\|_{\eh}\|x\|$$ for all $\alpha$,
and hence $\Phi_{u_{\alpha}}(x)\to \Phi_u(x)$ in the weak* topology.

It remains to show that $u_{\alpha}\to u$ in the weak* topology of $\ev(G)$.
Let $\phi, \psi\in A(G)$, viewed as (weak* continuous) functionals on $\vn(G)$,
and let $\xi,\xi',\eta,\eta'\in L^2(G)$ be such that
$\phi(s) = (\lambda_s\xi,\eta)$ and $\psi(s) = (\lambda_s\xi',\eta')$, $s\in G$.
Write
$A_{\alpha} = (a_i^{\alpha})_{i=1}^{\infty}$ and $B_{\alpha} = (b_i^{\alpha})_{i=1}^{\infty}$.
Then, as pointed out in \cite{er},
$$\langle A\odot B, \phi\otimes \psi\rangle = \sum_{i=1}^{\infty}  \langle a_i,\phi\rangle \langle b_i,\psi\rangle
\mbox{ and }
\langle A_{\alpha}\odot B_{\alpha}, \phi\otimes \psi\rangle =
\sum_{i=1}^{\infty}  \langle a_i^{\alpha},\phi\rangle \langle b_i^{\alpha},\psi\rangle.$$
Thus, using (\ref{eq_vndu}), we obtain
\begin{eqnarray*}
& &
|\langle A_{\alpha}\odot B_{\alpha}, \phi\otimes \psi\rangle - \langle A\odot B, \phi\otimes \psi\rangle|\\
& \leq &
\sum_{i=1}^{\infty}  \left|(a_i^{\alpha}\xi,\eta)(b_i^{\alpha}\xi',\eta') -
(a_i\xi,\eta)(b_i\xi',\eta')\right|\\
& \leq &
\sum_{i=1}^{\infty}  \left|(a_i^{\alpha}\xi,\eta)(b_i^{\alpha}\xi',\eta') -
(a_i^{\alpha}\xi,\eta)(b_i\xi',\eta')\right|\\
& + &
\sum_{i=1}^{\infty}  \left|(a_i^{\alpha}\xi,\eta)(b_i\xi',\eta') -
(a_i\xi,\eta)(b_i\xi',\eta')\right|\\
& = &
\sum_{i=1}^{\infty}  |(a_i^{\alpha}\xi,\eta)||(b_i^{\alpha}\xi',\eta') - (b_i\xi',\eta')|
+
\sum_{i=1}^{\infty}  |(b_i\xi',\eta')| |(a_i^{\alpha}\xi,\eta) - (a_i\xi,\eta)|\\
& \leq &
\|\eta\|\|\eta'\| \sum_{i =1}^{\infty}  \|a_i^{\alpha}\xi\|\|b_i^{\alpha}\xi' - b_i\xi'\|
+ \|\eta\|\|\eta'\|\sum_{i =1}^{\infty}  \|b_i\xi'\| \|a_i^{\alpha}\xi - a_i\xi\|\\
& \leq &
\|\eta\|\|\eta'\| \left(\sum_{i =1}^{\infty}  \|a_i^{\alpha}\xi\|^2\right)^{1/2}
\left(\sum_{i=1}^{\infty}  \|b_i^{\alpha}\xi' - b_i\xi'\|^2\right)^{1/2}\\
& + &
\|\eta\|\|\eta'\| \left(\sum_{i=1}^{\infty}  \|b_i\xi'\|^2\right)^{1/2}
\left(\sum_{i=1}^{\infty}  \|a_i^{\alpha}\xi - a_i\xi\|^2\right)^{1/2}\\
& = &
\|\eta\|\|\eta'\| \|A_{\alpha}\xi\| \|B_{\alpha}\xi' - B\xi'\|
+ \|\eta\|\|\eta'\| \|B\xi'\| \|A_{\alpha}\xi - A\xi\|.
\end{eqnarray*}
It follows that
$$\langle A_{\alpha}\odot B_{\alpha}, \phi\otimes \psi\rangle - \langle A\odot B, \phi\otimes \psi\rangle
\to\mbox{}_{\alpha} \ 0.$$
Since $A(G)\odot A(G)$ is dense in $\ahg$ and
the family $(A_{\alpha}\odot B_{\alpha})_{\alpha}$ of functionals
on $\ahg$ is uniformly bounded,
we conclude that
$$\langle A_{\alpha}\odot B_{\alpha}, w\rangle - \langle A\odot B, w\rangle
\to\mbox{}_{\alpha} \ 0$$
for every $w\in \ahg$, that is,
$u_{\alpha}\to u$ in the weak* topology of $\ev(G)$.
\end{proof}

In the sequel, we write $\cl K = \cl K(L^2(G))$.

\begin{lemma}\label{l_compactco}
Let $(u_{\alpha})_{\alpha}\subseteq \ev(G)$ be a uniformly bounded net, converging in the
weak* topology to an element $u\in \ev(G)$. Then
$$(\Phi_{u_{\alpha}}(T)\xi,\eta)\to\mbox{}_{\alpha} (\Phi_{u}(T)\xi,\eta),$$
for all $T\in \cl K$ and all $\xi,\eta\in L^2(G)$.
\end{lemma}
\begin{proof}
Fix $\xi,\eta\in L^2(G)$, and assume first that
$T = f\otimes g^*$ where $f,g\in L^2(G)$.
Let $\phi$ (resp. $\psi$) be the restriction of the vector functional
$\omega_{f,\eta}$ (resp. $\omega_{\xi,g}$) to
$\vn(G)$, viewed as an element of $A(G)$.
Suppose that $u = \sum_{i = 1}^{\infty}  a_i\otimes b_i$ is a w*-representation of $u$.
Then
\begin{eqnarray}\label{eq_ux}
(\Phi_{u}(T)\xi,\eta)
& = &
(\Phi_{u}(f\otimes g^*)\xi,\eta) = \sum_{i=1}^{\infty} (a_i(f\otimes g^*)b_i\xi,\eta)\nonumber\\
& = &
\sum_{i=1}^{\infty}  (((a_i f)\otimes (b_i^*g)^*)\xi,\eta)
= \sum_{i=1}^{\infty}  (\xi,b_i^*g) (a_i f,\eta)\\
& = &
\sum_{i=1}^{\infty}  (b_i\xi,g) (a_i f,\eta)
= \sum_{i=1}^{\infty}  \langle a_i,\phi\rangle \langle b_i,\psi\rangle
= \langle u,\phi\otimes\psi\rangle,\nonumber
\end{eqnarray}
where the last equality follows from \cite[(5.9)]{er}.
Since $u_{\alpha}\to_{\alpha} u$ in the weak* topology, we conclude that
$$(\Phi_{u_{\alpha}}(T)\xi,\eta) \to\mbox{}_{\alpha} (\Phi_{u}(T)\xi,\eta),$$
whenever $T$ has rank one. By linearity, the convergence holds
for any finite rank operator $T$.

Assume that $T\in \cl K$ is arbitrary and let $\epsilon > 0$.
Suppose that $C > 0$ is such that $\|u\|_{\eh}\leq C$ and $\|u_{\alpha}\|_{\eh}\leq C$ for all $\alpha$,
and choose a finite rank operator $T_0$ with
$\|T - T_0\| < \frac{\epsilon}{3C\|\xi\|\|\eta\|}$.
Let $\alpha_0$ be such that $|(\Phi_{u_{\alpha}}(T_0)\xi,\eta) - (\Phi_{u}(T_0)\xi,\eta)| < \frac{\epsilon}{3}$
for every $\alpha\geq \alpha_0$.
If $\alpha\geq \alpha_0$ then
\begin{eqnarray*}
& &
|(\Phi_{u_{\alpha}}(T)\xi,\eta) - (\Phi_{u}(T)\xi,\eta)|\\
& \leq &
|(\Phi_{u_{\alpha}}(T)\xi,\eta) - (\Phi_{u_{\alpha}}(T_0)\xi,\eta)|
 +
|(\Phi_{u_{\alpha}}(T_0)\xi,\eta) - (\Phi_{u}(T_0)\xi,\eta)|\\
&  + &
|(\Phi_{u}(T_0)\xi,\eta) - (\Phi_{u}(T)\xi,\eta)|
\leq
2C \|\xi\|\|\eta\| \|T - T_0\| + \epsilon/3  < \epsilon.
\end{eqnarray*}
\end{proof}

In the next statement, we use Proposition \ref{p_algtenin} to identify
$M^{\cb}A(G)\odot M^{\cb}A(G)$ with a subspace of $M^{\cb}\ahg$.

\begin{proposition}\label{p_meastp}
Let $\mu,\nu\in M(G)$ and $u = \lambda(\mu)\otimes\lambda(\nu)\in \ev(G)$.
For all $w \in \overline{M^{\cb}A(G)\odot M^{\cb}A(G)}^{\|\cdot\|_{\rm cbm}}$,
all $T\in \cl B(L^2(G))$ and all $\xi,\eta\in L^2(G)$,
the function $(s,t)\to w(s,t) (\lambda_s X \lambda_t\xi,\eta)$ is
$|\mu|\times|\nu|$-integrable and
\begin{equation}\label{eq_wdotu}
(\Phi_{w\cdot u}(T)\xi,\eta) = \int_{G\times G} w(t,s) (\lambda_t T \lambda_s\xi,\eta) d\mu(t)d\nu(s).
\end{equation}
\end{proposition}
\begin{proof}
Fix $T\in \cl B(L^2(G))$ and $\xi,\eta\in L^2(G)$. For $w\in \cbmh$, set
$$\nph_w(t,s) = w(t,s) (\lambda_t T \lambda_s\xi,\eta), \ \ \ t,s\in G.$$
Using (\ref{eq_unif}), we have
\begin{eqnarray*}
\int_{G\times G} |\varphi_w| d(|\mu|\times |\nu|)
& \leq & \|w\|_{\rm cbm} \int_{G\times G} | (T \lambda_t\xi,\lambda_{s^{-1}}\eta)| d|\mu|(s) d|\nu|(t)\\
& \leq & \|w\|_{\rm cbm} \|T\| \int_{G\times G} \|\lambda_t\xi\| \|\lambda_{s^{-1}}\eta\|d|\mu|(s) d|\nu|(t)\\
& = & \|w\|_{\rm cbm} \|T\| \int_G \|\lambda_t\xi\|d|\nu|(t) \int_G\|\lambda_{s^{-1}}\eta\|d|\mu|(s) \\
& \leq & \|w\|_{\rm cbm} \|T\| \|\xi\|\|\eta\| |\mu|(G) |\nu|(G).
\end{eqnarray*}
Since $\mu$ and $\nu$ are complex measures, they have finite total variation, and hence
$\nph_w\in L^1(G\times G,|\mu|\times |\nu|)$.

Let $\phi,\psi\in M^{\cb}A(G)$ and $w = \phi\otimes\psi$.
For every $\zeta\in L^2(G)$ we have
\begin{eqnarray*}
\int_G (T\lambda(\psi\nu)\xi)(r)\overline{\zeta(r)}dr & = &
(T\lambda(\psi\nu)\xi,\zeta) = (\lambda(\psi\nu)\xi,T^*\zeta)\\
& = & \int_G (\psi(s)\lambda_s\xi, T^*\zeta) d\nu(s)\\
& = &  \int_G (\psi(s)T\lambda_s\xi, \zeta) d\nu(s)\\
& = & \int_G \int_G\psi(s)(T\lambda_s\xi)(r)\overline{\zeta(r)} dr d\nu(s)\\
& = & \int_G \left(\int_G\psi(s)(T\lambda_s\xi)(r)d\nu(s) \right) \overline{\zeta(r)} dr.
\end{eqnarray*}
It follows that
\begin{equation}\label{eq_xr}
(T\lambda(\psi\nu)\xi)(r) = \int_G \psi(s) (T\lambda_s\xi)(r) d\nu(s), \ \ \mbox{ for almost all } r\in G.
\end{equation}
By (\ref{eq_xr}) and Lemma \ref{l_elt},
\begin{eqnarray*}
(\Phi_{w\cdot u}(T)\xi,\eta)
& = & (\Phi_{\lambda(\phi\mu)\otimes \lambda(\psi\nu)}(T)\xi,\eta)
= (T \lambda(\psi\nu)\xi,\lambda(\phi\mu)^*\eta)\\
& = &
\int_G\int_{G\times G} \psi(s) (T\lambda_s\xi)(r)\phi(t)
\overline{(\lambda_{t^{-1}}\eta)(r)} d\nu(s)d\mu(t)dr\\
& = & \int_{G\times G} \nph_w(t,s)  d\mu(t)d\nu(s).
\end{eqnarray*}
By linearity, (\ref{eq_wdotu}) holds for every $w \in M^{\cb}A(G)\odot M^{\cb}A(G)$.

Now suppose that $w$ is in the closure of
$M^{\cb}A(G)\odot M^{\cb}A(G)$, and let
$(w_k)_{k\in \bb{N}}$ be a sequence in $M^{\cb}A(G)\odot M^{\cb}A(G)$ such that
$\|w_k - w\|_{\rm cbm} \to_{k\to \infty} 0$.
By Proposition \ref{p_cow}, $w_k\cdot u\to w\cdot u$ in the norm of $\ev(G)$;
thus, $\Phi_{w_k\cdot u}\to \Phi_{w\cdot u}$ in the completely bounded norm, and hence
\begin{equation}\label{eq_Phi}
(\Phi_{w_k\cdot u}(T)\xi,\eta) \to_{k\to \infty} (\Phi_{w\cdot u}(T)\xi,\eta).
\end{equation}

On the other hand, by (\ref{eq_unif}), $w_k \rightarrow w$ pointwise.
By the first paragraph of the proof,
the function $(t,s)\to (\lambda_t T \lambda_s\xi,\eta)$ is $|\mu|\times|\nu|$-integrable.
It follows that the functions $\nph_{w_k}$, $k\in \bb{N}$, are dominated pointwise
by an integrable function. Now the Lebesgue Dominated Convergence Theorem implies that
\[
\int_{G\times G} w_k(t,s) (\lambda_t T \lambda_s\xi,\eta) d\mu(t) d\nu(s)
\to \int_{G\times G} w(t,s) (\lambda_t T \lambda_s\xi,\eta) d\mu(t) d\nu(s)
\]
and this, together with (\ref{eq_Phi}) and the second paragraph of the proof, shows that
(\ref{eq_wdotu}) holds for the function $w$.
\end{proof}

If $f\in L^{\infty}(G)$, let $M_f$ be the operator on $L^2(G)$ of multiplication by $f$,
and set $\cl D = \{M_f : f\in L^{\infty}(G)\}$.
Recall that, for $\mu\in M(G)$, we let $\Theta(\mu)$ be the completely bounded linear map
on $\cl B(L^2(G))$ given by
\[
\Theta(\mu)(T) = \int_G \lambda_s T \lambda_s^* d\mu(s), \ \ \ T\in \cl B(L^2(G)),
\]
where the intergral is understood in the weak sense.
It is well-known that
$\Theta$ maps $M(G)$ onto the space of all
completely bounded weak* continuous $\vn(G)'$-bimodule maps
that leave $\cl D$ invariant (see \cite{neufang_thesis} and \cite[Theorem~3.2]{nrs}).

We recall some notions from \cite{a} and \cite{eks}.
A measurable set
$\kappa\subseteq G\times G$ is called  \emph{marginally null}
if there exists a null set $M\subseteq G$ such that
$\kappa\subseteq (M\times G)\cup (G\times M)$.
Two measurable sets $\kappa_1$ and $\kappa_2$ of $G\times G$ are called
\emph{$\omega$-equivalent} (written $\kappa_1\cong \kappa_2$)
if their symmetric difference is marginally null;
we say that $\kappa_1$ is \emph{marginally contained} in $\kappa_2$
if $\kappa_1\setminus \kappa_2$ is marginally null.
The set $\kappa$ is called \emph{$\omega$-open} if it is $\omega$-equivalent to a
countable union of sets of the form $\alpha\times\beta$, where $\alpha, \beta\subseteq G$
are measurable. It is called \emph{$\omega$-closed} if its complement is $\omega$-open.
For measurable $\alpha \subseteq G$, let $P(\alpha) \in \cl D$ be the
projection given by multiplication by the characteristic function of $\alpha$.
An operator $T\in \cl B(L^2(G))$ is said to be \emph{supported by} $\kappa$
if $P(\beta)TP(\alpha) = 0$ whenever $\alpha,\beta\subseteq G$ are
measurable sets with $(\alpha\times\beta)\cap \kappa \cong \emptyset$.
The \emph{$\omega$-support} of a subset $\cl U\subseteq \cl B(L^2(G))$
is the smallest (with respect to marginal containment) $\omega$-closed set $\kappa$
such that every element of $\cl U$ is supported by $\kappa$.
It was shown in \cite{a} and \cite{eks} that,
given any $\omega$-closed set $\kappa$, there exist a largest weak* closed $\cl D$-bimodule
$\frak{M}_{\max}(\kappa)$ (namely, the space of all
operators on $L^2(G)$ supported by $\kappa$) and a smallest weak* closed $\cl D$-bimodule
$\frak{M}_{\min}(\kappa)$ with support $\kappa$.

We recall that for $E \subseteq G$, we write 
$E^* = \{(s,t) \in G\times G:\ ts^{-1} \in E\}$ and $E^\sharp = \{(s,t) \in G\times G: \ st\in E\}$.

\begin{proposition}\label{p_apabo}
Let $G$ be a second countable locally compact group, $E\subseteq G$ be a compact set
and $V\subseteq G$ be a compact neighbourhood of $e$. Then
$$\frak{M}_{\max}(E^*) \subseteq \overline{\frak{M}_{\max}((EV)^*)\cap \cl K}^{w^*}.$$
\end{proposition}
\begin{proof}
Let $U\subseteq G$ be an open set such that $e\in U$ and $\overline{U} = V$.
The set $EU$ is open, and hence $(EU)^*$ is $\omega$-open.
Write $(EU)^* \cong \cup_{i=1}^{\infty} \alpha_i\times \beta_i$, where
$\alpha_i,\beta_i\subseteq G$ are measuarble subsets.
Let $\epsilon > 0$. By \cite[Lemma 3.4]{eks}, there exist $l_{\epsilon}\in \bb{N}$ and a measurable subset $L_{\epsilon}\subseteq G$
such that $|L_{\epsilon}^c| < \epsilon$ and
\begin{equation}\label{eq_lep}
(EU)^*\cap (L_{\epsilon} \times L_{\epsilon}) \subseteq \cup_{i=1}^{l_{\epsilon}} \alpha_i\times\beta_i.
\end{equation}
It is easy to see that there exist (finite) families $\{\sigma_p\}_{p=1}^M$ and $\{\tau_q\}_{q=1}^N$
of pairwise disjoint measurable subsets of $G$
and a subset $R\subseteq \{1,\dots,M\}\times\{1,\dots,N\}$
such that
\begin{equation}\label{eq_epq}
\cup_{i=1}^{l_{\epsilon}} \alpha_i\times\beta_i = \cup_{(p,q)\in R} \sigma_p\times\tau_q.
\end{equation}
For each $p,q\in \{1,\dots,M\}\times\{1,\dots,N\}$, let
$\Pi_{p,q} : \cl B(L^2(G))\to \cl B(L^2(G))$ be the idempotent given by
$$\Pi_{p,q}(T) = P(\tau_q)TP(\sigma_p), \ \ \ T\in \cl B(L^2(G)).$$
By (\ref{eq_lep}) and (\ref{eq_epq}),
$$T = \sum_{(p,q)\in R} \Pi_{p,q}(T), \ \mbox{ for every } T\in \frak{M}_{\max}(E^*\cap (L_{\epsilon} \times L_{\epsilon})).$$
However, $\Pi_{p,q}(T)\in \frak{M}_{\max}(\sigma_p\times\tau_q)$,
while
$$\frak{M}_{\max}(\sigma_p\times\tau_q)
\subseteq \overline{\frak{M}_{\max}(\sigma_p\times\tau_q)\cap \cl K}^{w^*}
\subseteq \overline{\frak{M}_{\max}((EV)^*)\cap \cl K}^{w^*},$$
whenever $(p,q)\in R$.
It follows that
$$\frak{M}_{\max}(E^*\cap (L_{\epsilon} \times L_{\epsilon}))
\subseteq \overline{\frak{M}_{\max}((EV)^*)\cap \cl K}^{w^*}.$$
On the other hand, 
$$\frak{M}_{\max}(E^*\cap (L_{\epsilon} \times L_{\epsilon})) 
= P(L_{\epsilon})\frak{M}_{\max}(E^*) P(L_{\epsilon})$$
and the claim follows after passing to a limit as $\epsilon \to 0$.
\end{proof}

\begin{lemma}\label{l_inmmax}
If $(\kappa_m)_{m\in \bb{N}}$ is a decreasing sequence of $\omega$-closed sets,
then $\cap_{m\in \bb{N}} \frak{M}_{\max}(\kappa_m) = \frak{M}_{\max}(\cap_{m\in \bb{N}}\kappa_m)$.
\end{lemma}
\begin{proof}
Set $\kappa = \cap_{m\in \bb{N}}\kappa_m$. Since $\kappa\subseteq \kappa_m$,
we have that $\frak{M}_{\max}(\kappa)\subseteq \frak{M}_{\max}(\kappa_m)$, $m\in \bb{N}$;
thus, $\frak{M}_{\max}(\kappa)\subseteq \cap_{m\in \bb{N}}\frak{M}_{\max}(\kappa_m)$.

Assuming that $\kappa_m^c \cong \cup_{k\in \bb{N}} \alpha_k^m\times\beta_k^m$,
where $\alpha_k^m$ and $\beta_k^m$ are measurable subsets of $G$,
we have that $\kappa^c \cong \cup_{k,m\in \bb{N}} \alpha_k^m\times\beta_k^m$.
Suppose that $\alpha$ and $\beta$ are measurable subsets of $G$ such that
$\kappa\cap (\alpha\times\beta)\cong\emptyset$ and $T\in \cap_{m\in \bb{N}} \frak{M}_{\max}(\kappa_m)$.
By deleting null sets from $\alpha$ and $\beta$ if necessary,
we can assume that $\alpha\times\beta \subseteq \kappa^c$.
Using \cite[Lemma 3.4]{eks}, one can see that
there exists an increasing sequence $(K_n)_{n\in \bb{N}}$
of compact subsets of $G$ such that $G\setminus (\cup_{n\in \bb{N}} K_n)$ is null, and $l_n\in \bb{N}$
such that
$$(\alpha\cap K_n)\times (\beta\cap K_n)\subseteq \cup_{k,m=1}^{l_n} \alpha_k^m\times\beta_k^m.$$
Using a decomposition analogous to (\ref{eq_epq}),
we conclude that $P(\beta\cap K_n) T P(\alpha\cap K_n) = 0$.
Since this holds for all $n\in \bb{N}$, we conclude after passing to a limit that
$P(\beta)T P(\alpha) = 0$.
This shows that $T\in \frak{M}_{\max}(\kappa)$ and the proof is complete.
\end{proof}

For subsets $E,K\subseteq G$, write $\Omega_{E,K} = \overline{\cup_{s\in K} s E s^{-1}}$.

\begin{theorem}\label{th_invars}
Let $G$ be a second countable locally compact group
and $E$ and $K$ be compact subsets of $G$.
If  $u \in \ev(G)$, $\supp_{\rm h} (u) \subseteq E^{\sharp} \cap (K\times K)$,
and $w_K\in A(G\times G)$ is supported in $K\times K$
then $\Phi_{w_K\cdot u}$ maps $\frak{M}_{\max}(E^*)$ into $\frak{M}_{\max}((\Omega_{E,K}E)^*)$.
\end{theorem}
\begin{proof}
Suppose that $\supp_{\rm h}(u) \subseteq E^{\sharp}\cap (K\times K)$.
Let $U$ and $V$ be compact symmetric neighbourhoods of the neutral element $e$,
and let $U_0$ be an open symmetric neighbourhood of $e$  and $W$ an open set such that
$\overline{U_0}\subseteq W \subseteq U$.
Let $w\in A(G\times G)$ be a function supported in $(EU)^{\sharp}$
such that $w|_{(EU_0)^{\sharp}\cap (K\times K)} \equiv 1$.

Fix $T\in \frak{M}_{\max}((EV)^*)\cap \cl K$.
We claim that
\begin{equation}\label{eq_Om}
\Phi_{w_K\cdot u}(T)\in \frak{M}_{\max}((\Omega_{EV,K}EU)^*).
\end{equation}
Suppose that $L, M \subseteq G$ are measurable sets with
$(L \times M)\cap (\Omega_{EV,K}EU)^* \cong \emptyset$.
By deleting null sets from $L$ and $M$ if necessary,
we may assume that, in fact, $(L \times M)\cap (\Omega_{EV,K}EU)^* = \emptyset$,
that is,
\begin{equation}\label{eq_ovku}
(M L^{-1}) \cap (\Omega_{EV,K}EU) = \emptyset.
\end{equation}

Let $\xi\in L^2(G)$ (resp. $\eta\in L^2(G)$) be a function that
vanishes almost everywhere on $L^c$ (resp. $M^c$).
By Lemma \ref{l_appr}, there exists a uniformly bounded net
$(u_{\alpha})_{\alpha}\subseteq \lambda(L^1(G))\odot \lambda(L^1(G))$ such that
$u_{\alpha}\to_{\alpha} u$ in the weak* topology of $\ev(G)$.
Write $\tilde{u}_{\alpha}$ for the function on $G\times G$ corresponding to
some representative of $u_{\alpha}$, so that $u_{\alpha} = (\lambda\otimes\lambda)(\tilde{u}_{\alpha})$. 
Recall that $w_K$ is an element of $A(G\times G)$ and hence of $\ahg$;
by Proposition \ref{p_algtenin}, $w_K$ belongs to the $\|\cdot\|_{\rm cbm}$-closure of ${M^{\cb}A(G)\odot M^{\cb}A(G)}$.
Using  Lemmas \ref{l_inter} and \ref{l_compactco}  as well as Propositions \ref{p_cow} and \ref{p_meastp},
we have
\begin{eqnarray*}
(\Phi_{w_K\cdot u}(T)\xi,\eta)
& = &
(\Phi_{w_Kw\cdot u}(T)\xi,\eta)
= \lim\mbox{}_{\alpha} (\Phi_{w_Kw\cdot u_\alpha}(T)\xi,\eta)\\
& = &
\lim\mbox{}_{\alpha} \int_{G\times G} w_K(s,t) w(s,t) (\lambda_s T \lambda_t\xi,\eta)\tilde{u}_{\alpha}(s,t)dsdt.
\end{eqnarray*}
Set
$$h(s,t) = (\lambda_s T \lambda_t\xi,\eta), \ \ s,t\in G.$$
We claim that the integrand $w_K w h \tilde{u}_{\alpha}$ is identically zero.
We consider the following cases:

\medskip

\noindent {\it Case 1.} $(s,t)\not\in K\times K$. In this case, $w_K(s,t) = 0$.

\medskip

\noindent {\it Case 2.} $(s,t)\not\in (EU)^{\sharp}$. In this case, $w(s,t) = 0$.

\medskip

\noindent {\it Case 3.} $(s,t)\in (K\times K)\cap (EU)^{\sharp}$. We claim that,
in this case $h(s,t) = 0$. To see this, note that $\supp(\lambda_t\xi)\subseteq tL$
and $\supp(\lambda_{s^{-1}}\eta)\subseteq s^{-1}M$.
Since $T$ is supported by $(EV)^*$, it suffices to see that
$(tL \times s^{-1}M) \cap (EV)^* = \emptyset$.
Assume, by way of contradiction, that
there exist $x,y\in G$ with
$$(x,y)\in (tL \times s^{-1}M) \cap (EV)^*.$$
Write $x = tx_0$ and $y = s^{-1} y_0$ for some $x_0\in L$ and
$y_0\in M$. Then  	
$$s^{-1}y_0 x_0^{-1}t^{-1} = yx^{-1} \in EV,$$
and so
$$y_0 x_0^{-1} \in s EV t.$$
On the other hand, $st\in EU$ and hence
$t\in s^{-1}EU$. Thus,
$$y_0 x_0^{-1} \in s EV s^{-1} EU\subseteq \Omega_{EV,K} EU.$$
This contradicts (\ref{eq_ovku}).

\medskip

It now follows that $(\Phi_{w_K\cdot u}(T)\xi,\eta) = 0$
whenever $\xi = P(L)\xi$ and $\eta = P(M)\eta$; this implies that
$\Phi_{w_K\cdot u}(T)\in \frak{M}_{\max}((\Omega_{EV,K}EU)^*)$.

Let $(U_k)_{k\in \bb{N}}$ be a sequence of compact neighbourhoods of $e$
such that $\cap_{k\in \bb{N}} U_k = \{e\}$.
Note that
\begin{equation}\label{eq_intom}
\cap_{k\in \bb{N}} \Omega_{EV,K}EU_k \subseteq \overline{\Omega_{EV,K}E} = \Omega_{EV,K}E.
\end{equation}
To see (\ref{eq_intom}), assume that $t\in \cap_{k\in \bb{N}} \Omega_{EV,K}EU_k$ 
and, for every $k$, write $t = s_k t_k$, where $s_k\in \Omega_{V,K}E$ and $t_k\in U_k$.
Then $t_k\to_{k\to \infty} e$ and hence $s_k\to t$. The inclusion in (\ref{eq_intom}) is thus proved;
the equality follows from the fact that $\Omega_{EV,K}E$ is compact.

By Lemma \ref{l_inmmax},
$\Phi_{w_K\cdot u}(T)\in \frak{M}_{\max}((\Omega_{EV,K}E)^*)$.
We have thus shown that
$$\Phi_{w_K\cdot u}(\frak{M}_{\max}((EV)^*)\cap \cl K)\subseteq \frak{M}_{\max}((\Omega_{EV,K}E)^*).$$
By Proposition \ref{p_apabo} and the weak* continuity of $\Phi_{w_K\cdot u}$ we have that
\begin{equation}\label{eq_d}
\Phi_{w_K\cdot u}(\frak{M}_{\max}(E^*))\subseteq \frak{M}_{\max}((\Omega_{EV,K}E)^*).
\end{equation}

Let $(V_k)_{k\in \bb{N}}$ be a sequence of compact neighbourhoods of $e$
such that $\cap_{k\in \bb{N}} V_k = \{e\}$.
We claim that
\begin{equation}\label{eq_vkin}
\cap_{k\in \bb{N}} \Omega_{EV_k,K}E = \Omega_{E,K}E.
\end{equation}
To show (\ref{eq_vkin}), assume that $t\in \cap_{k\in \bb{N}} \Omega_{EV_k,K}E$.
For each $k\in \bb{N}$, there exist $s_k\in K$, $r_k, p_k\in E$ and $t_k\in V_k$ such that
$s_k r_kt_k s_k^{-1} p_k =  t$. Then $t_k\to e$ and by the compactness of $E$ and $K$ we have that
$t\in \Omega_{E,K}E$;  (\ref{eq_vkin}) is hence proved.
It now follows from (\ref{eq_d}) and Lemma \ref{l_inmmax} that
$$\Phi_{w_K\cdot u}(\frak{M}_{\max}(E^*))\subseteq \frak{M}_{\max}((\Omega_{E,K}E)^*).$$
\end{proof}

In the following corollaries, if $H$ is a closed subgroup of $G$,
we write $\vn(H)$ for the von Neumann subalgebra of $\vn(G)$ generated by $\lambda_s$, $s\in H$.

\begin{corollary}\label{c_no1}
Let $G$ be a second countable locally compact group,
$H\subseteq G$ be a compact
normal subgroup and $K\subseteq G$ be a compact subset.
If $\supp_h(u) \subseteq H^{\sharp} \cap (K\times K)$
then $\Phi_u$ leaves the von Neumann algebra $\cl M_H$ generated by $\vn(H)$ and $\cl D$ invariant.
\end{corollary}
\begin{proof}
Under the stated assumptions, $\Omega_{H,K}\subseteq H$ and hence $\Omega_{H,K}H\subseteq HH\subseteq H$.
Suppose that $\supp_h(u) \subseteq H^\sharp \cap (K \times K)$ and
let $w_{\tilde{K}} \in \ahg$ be such that $w_{\tilde{K}} =1$ on $K \times K$ and
$\supp(w_{\tilde{K}}) \subseteq \tilde{K}$ for some compact set $\tilde{K}$ containing a neighbourhood of $K\times K$.
Then $w_{\tilde{K}}. u =u$ and, by Theorem~\ref{th_invars},
\[
\Phi_u(\frak{M}_{\max}(H^*)) = \Phi_{w_{\tilde{K}} \cdot u}(\frak{M}_{\max}(H^*)) \subseteq \frak{M}_{\max}(H^*).
\]
The claim now follows from the fact that  $\frak{M}_{\max}(H^*) = \cl M_H$
(see \cite{akt}).
\end{proof}

\begin{corollary}\label{c_no2}
Let $G$ be a second countable compact group
and $H\subseteq G$ be a closed normal subgroup.
If $\supp(u) \subseteq H^{\sharp}$
then $\Phi_u$ leaves the von Neumann algebra $\cl M_H$ generated by $\vn(H)$ and $\cl D$ invariant.
\end{corollary}
\begin{proof}
Immediate from Corollary \ref{c_no1}.
\end{proof}

\begin{proposition}\label{p_mv}
Let $u\in \ev(G)$ and $\mu\in M(G)$
be such that $\Phi_u = \Theta(\mu)$. Then
$$\langle u,v\rangle = \int_{G} v(s,s^{-1}) d\mu(s), \ \ \mbox{ for all }\ v\in \ahg.$$
\end{proposition}
\begin{proof}
Let $\xi,\eta, f, g \in L^2(G)$ and 
$\phi(s) = (\lambda_s f,\eta)$ and $\psi(s)=(\lambda_s \xi, g)$ ($s \in G$). Then by (33) we have
\begin{eqnarray*}
\langle u, \phi \otimes \psi\rangle = (\Phi_u(f \otimes g^*) \xi, \eta)
&=& \langle \Theta(\mu)(f \otimes g^*) \xi, \eta\rangle \\
&=& \left(\left(\int_G \lambda_s (f\otimes g^*) \lambda_s^* d\mu(s)\right)\xi,\eta\right)\\
& = &
\left(\left(\int_G (\lambda_s f)\otimes (\lambda_s g)^* d\mu(s)\right)\xi,\eta\right)\\
& = &
\int_G (\lambda_s f,\eta)(\xi, \lambda_s g) d\mu(s)\\
&=& \int_G \phi(s) \psi(s^{-1}) d\mu(s)\\
&=& \int_G (\phi \otimes \psi)(s,s^{-1}) d\mu(s).
\end{eqnarray*}
On the other hand, if $v\in \ahg$ then, by (\ref{eq_unif}),
$$\left|\int_{G} v(s,s^{-1}) d\mu(s)\right| \leq \|\mu\|\|v\|_{\infty} \leq
\|\mu\|\|v\|_{\hh}.$$
Thus, there exists a (unique) $u'\in \ev(G)$ such that
\begin{equation}\label{eq_vssi}
\langle u',v\rangle = \int_{G} v(s,s^{-1}) d\mu(s), \ \ \ v\in \ahg.
\end{equation}
As $A(G) \odot A(G)$ is dense in $A_h(G)$ and $\langle u,v\rangle = \langle u', v\rangle$ for every $v\in A(G)\odot A(G)$, we obtain the statement.
\end{proof}

We recall that $\tilde{\Delta} = \{(s,s^{-1}) : s\in G\}$ is the antidiagonal of $G$.

\begin{theorem}\label{th_antid}
Let $G$ be a second countable weakly amenable locally compact group
and $u\in \ev(G)$.
The following conditions are equivalent:

(i) \ $\supp_{\hh}(u) \subseteq \tilde{\Delta}$;

(ii) there exists $\mu\in M(G)$ such that $\Phi_u = \Theta(\mu)$.
\end{theorem}
\begin{proof}
(i)$\Rightarrow$(ii)
Suppose that $\supp_{\hh}(u)\subseteq \tilde{\Delta}$.
For a compact set $K\subseteq G$ and
$w\in A(G\times G)$ with support in $K\times K$, we have,
by Lemma \ref{l_supp-equality}, that $\supp_{\hh}(w\cdot u)\subseteq \tilde{\Delta}\cap K\times K$.
Hence, by Corollary \ref{c_no1}, 
$$\Phi_{w\cdot u}(\cl D)\subseteq \cl D;$$
by \cite[Theorem~3.2]{nrs}, there exists a (unique) measure $\mu_{K,w}\in M(G)$
such that
\begin{equation}\label{eq_walphm}
\Phi_{w\cdot u} = \Theta(\mu_{K,w}).
\end{equation}

Since $G$ is weakly amenable, by Theorem~\ref{p_HWA},
there exist a constant $C>0$, compact sets $K_{\alpha}\subseteq G$
 and a  net $(w_\alpha)_{\alpha}$ of elements in
$ A(G) \odot  A(G)$ supported in $K_{\alpha}\times K_{\alpha}$
such that $\|w_\alpha\|_{\rm cbm}\leq C$ for all $\alpha$ and
$w_\alpha v\to v$ in $\ahg$ for all $v\in \ahg$.
Set $\mu_{\alpha} = \mu_{K_{\alpha},w_{\alpha}}$;
then, by Proposition~\ref{p_cow},
$$\|\mu_{\alpha}\| = \|\Theta(\mu_{\alpha})\|_{\cb} =
\|w_{\alpha}\cdot u\|_{\eh}\leq C\|u\|_{\eh}$$
for all $\alpha$.
Thus, the net $(\mu_{\alpha})_{\alpha}$ has a weak* cluster point;
we assume without loss of generality that
$\mu_{\alpha}\to_{\alpha}\mu$ in the weak* topology of $M(G)$.
Let $f,g,\xi,\eta\in L^2(G)$. Then
the functions $s\to (\lambda_{s^{-1}}\xi,g)$ and $s\to (f,\lambda_{s^{-1}}\eta)$
belong to $C_0(G)$ and
\begin{eqnarray*}
(\Theta(\mu_{\alpha})(f\otimes g^*)\xi,\eta)
& = &
\int_G(\lambda_s(f\otimes g^*)\lambda_s^*\xi,\eta)d\mu_{\alpha}(s)\\
& = &
\int_G((f\otimes g^*)(\lambda_{s^{-1}}\xi),\lambda_{s^{-1}}\eta)d\mu_{\alpha}(s)\\
& = &
\int_G (\lambda_{s^{-1}}\xi,g)(f,\lambda_{s^{-1}}\eta)d\mu_{\alpha}(s)\\
& \to_{\alpha} &
\int_G (\lambda_{s^{-1}}\xi,g)(f,\lambda_{s^{-1}}\eta)d\mu(s)\\
& = &
(\Theta(\mu)(f\otimes g^*)\xi,\eta).
\end{eqnarray*}
It follows that
\begin{equation}\label{eq_finr}
(\Theta(\mu_{\alpha})(T)\xi,\eta) \to_{\alpha} (\Theta(\mu)(T)\xi,\eta)
\end{equation}
whenever $T$ is an operator of finite rank.

On the other hand,
$w_{\alpha}v\to v$ for every $v\in \ahg$ and hence
$w_{\alpha}\cdot u\to u$ in the weak* topology of $\ev(G)$.
By (\ref{eq_iw}) and Lemma \ref{l_compactco},
$$(\Phi_{w_{\alpha}\cdot u}(T)\xi,\eta)\to_{\alpha} (\Phi_{u}(T)\xi,\eta),$$
for all $T\in \cl K$.
It now follows from (\ref{eq_finr}) and (\ref{eq_walphm}) that
$$\Phi_{u}(T) = \Theta(\mu)(T),$$
for every finite rank operator $T$. Since $\Phi_{u}$ and $\Theta(\mu)$ are
weak* continuous, we conclude that $\Phi_u = \Theta(\mu)$.

(ii)$\Rightarrow$(i) Suppose that $v\in \ahg$ vanishes on $\tilde{\Delta}$.
By Proposition \ref{p_mv},
$$\langle u,v\rangle = \int_{G} v(s,s^{-1}) d\mu(s) = 0.$$
It follows that $\supp(u)\subseteq \tilde{\Delta}$.
\end{proof}

\noindent {\bf Remark } Since the singletons are sets of synthesis
and $\tilde{\Delta} = \{e\}^{\sharp}$,
Theorem \ref{t:Moore-groups} implies that, if $G$ is a Moore group then
$\tilde{\Delta}$ is a set of spectral synthesis for $\ahg$.
Theorem \ref{th_antid} shows that this holds true for any
weakly amenable second countable locally compact group.
In fact, the result shows that $\tilde{\Delta}$ satisfies the stronger condition
of being a Helson set.


\begin{thebibliography}{99}

\bibitem{akt}
\textsc{M. Anoussis, A. Katavolos and I. G. Todorov},
\textit{Ideals of the Fourier algebra, supports and harmonic operators},
\textrm{Math. Proc. Cambridge Philos. Soc. 161 (2016), 223-235}.


\bibitem{a}
{\sc W. B. Arveson}, {\it Operator algebras and invariant subspaces},
{\rm Ann. Math. (2) 100 (1974), 433-532}.


\bibitem{blm}
{\sc D. P. Blecher and C. Le Merdy}, {\it Operator algebras and their
modules -- an operator space approach}, {\rm Oxford University
Press, 2004}.


\bibitem{bs}
{\sc D. P. Blecher and R. R. Smith},
{\it The dual of the Haagrup tensor product},
{\rm J. London Math. Soc. (2) 4 (1992), 126-144.}


\bibitem{bo}
\textsc{N. P. Brown and N. Ozawa},
{\it C*-algebras and finite dimensional approximations},
{\rm American Mathematical Society, 2008}.


\bibitem{ch}
\textsc{J. de Canniere and U. Haagerup},
\textit{Multipliers of the Fourier algebras of some simple Lie groups and their discrete subgroups},
\textrm{Amer. J. Math. 107 (1985), no. 2, 455-500}.

\bibitem{coh}
\textsc{M. Cowling and U. Haagerup},
\textit{Completely bounded multipliers of the Fourier algebra of a
              simple {L}ie group of real rank one},
\textrm{Invent. Math. 96 (1989), no. 3, 507--549}.

\bibitem{daws}
\textsc{M. Daws},
\textit{Multipliers, self-induced and dual Banach algebras},
\textrm{Dissertationes Math. (Rozprawy Mat.) 470 (2010), 62pp}.

\bibitem{dix}
{\sc J. Dixmier},
{\it{$C\sp*$}-algebras},
{\rm North-Holland, 1977.}

\bibitem{er_book}
{\sc E. G. Effros and Z.-J. Ruan},
{\it Operator spaces},
{\rm Oxford University Press, 2000}.

\bibitem{er}
{\sc E. G. Effros and Z.-J. Ruan},
{\it Operator spaces tensor products and Hopf convolution algebras},
\textrm{ J. Operator Theory 50 (2003), 131-156}.

\bibitem{eks}
{\sc J. A. Erdos, A. Katavolos and V. S. Shulman},
{\it Rank one subspaces of bimodules over maximal abelian
selfadjoint algebras},
{\rm J. Funct. Anal. 157 (1998) no. 2, 554-587}.


\bibitem{eymard}
{\sc  P. Eymard}, {\it L'alg\`{e}bre de Fourier d'un groupe localement
compact}, {\rm Bull. Soc. Math. France 92 (1964), 181-236}.

\bibitem{fr}
{\sc B. E. Forrest and V. Runde},
{\it Amenability and weak amenability of the Fourier algebra},
{\rm Math. Z. 250 (2005), no. 4, 731-744}.

\bibitem{fss}
{\sc B. E. Forrest, E. Samei, and N. Spronk},
{\it Weak amenability of {F}ourier algebras on compact groups},
{\rm Indiana Univ. Math. J. 58 (2009), 1379-1393}.

\bibitem{gh}
{\sc F. Ghahramani}, {\it Isometric representation of {$M(G)$} on {$B(H)$}},
 {\rm Glasgow Math. J. 23 (1982) no. 2, 119-122}.


\bibitem{haag}
\textsc{U. Haagerup}, \textit{Decomposition of completely bounded
maps on operator algebras}, \textrm{unpublished manuscript}.

\bibitem{hr1}
{\sc  E. Hewitt and K. A. Ross},
{\it Abstract harmonic analysis I},
{\rm Springer-Verlag, 1979}.


\bibitem{hr2}
{\sc  E. Hewitt and K. A. Ross},
{\it Abstract harmonic analysis II},
{\rm Springer-Verlag, 1970}.

\bibitem{hnr}
{\sc  Zh. Hu, M. Neufang and Zh.-J. Ruan},
{\it Completely bounded multipliers over locally compact quantum groups},
{\rm Proc. Lond. Math. Soc. (3) 103 (2011), no. 1, 1-39}.

\bibitem{It}
{\sc  T. Itoh},
{\it The Haagerup type cross norm on $C^*$-algebras},
{\rm Proc. Amer. Math. Soc. 109 (1990), no. 3, 689-695}.


\bibitem{jnr}
{\sc M. Junge, M. Neufang and Zh.-J. Ruan},
{\it A representation theorem for locally compact quantum groups}, 
{\rm Internat. J. Math. 20 (2009), no. 3, 377-400}.


\bibitem{kaniuth}
{\sc E. Kaniuth},
{\it A course in commutative {B}anach algebras},
{\rm Springer, 2009}.


\bibitem{lt}
{\sc J. Ludwig and L. Turowska},
{\it On the connection between sets of operator synthesis and sets
              of spectral synthesis for locally compact groups},
{\rm J. Funct. Anal. 233 (2006), 206-227}.
		

\bibitem{moore}
 {\sc C. C. Moore},
{\it Groups with finite dimensional irreducible representations},
{\rm Trans. Amer. Math. Soc. 166 (1972), 401-410}.


\bibitem{neufang_thesis}
 {\sc M. Neufang}, 
 {\it Abstrakte harmonische analyse und modulhomomorphismen \"{u}ber von Neumann-algebren}, 
 {\rm PhD Thesis, Universit\"{a}t des Saarlandes, 2000}.
 

\bibitem{nrs}
{\sc M. Neufang,  Z-J.  Ruan and N. Spronk},
{\it Completely isometric representations of $M_{cb}A(G)$ and  $UCB(\widehat G)^*$},
{\rm Trans. Am. Math. Soc. 360 (2008), no. 3, 1133-1161}.

\bibitem{palmer}
{\sc T.  W. Palmer},
{\it Banach algebras and the general theory of *-algebras. Vol. 2},
{\rm Cambridge University Press, 2001}.
		

\bibitem{rs}
{\sc M. Rostami and N. Spronk},
{\it Convolutions of the Haagerup tensor product of Fourier algebras},
{\rm Houston J. Math. 42 (2016), 597-611}.


\bibitem{stt}
{\sc V. S. Shulman, I. G. Todorov and L. Turowska},
{\it Sets of multiplicity and closable multipliers of group algebras},
{\rm J. Funct. Anal. 268 (2015), 1454-1508}.


		
		
\bibitem{st}
{\sc N. Spronk and L. Turowska},
{\it Spectral synthesis and operator synthesis for compact groups},
{\rm J. London Math. Soc. (2) 66 (2002), 361-376}.


\bibitem{stor}
{\sc E. St\o{}rmer},
{\it Regular abelian Banach algebras of linear maps of operator algebras},
{\rm J. Funct. Anal. 37 (1980), 331-373}.

\bibitem{t_serdika}
{\sc I. G. Todorov},
{\it Interactions between harmonic analysis and operator theory},
{\rm   Serdica Math. J. 41 (2015), no. 1, 13-34}.


\bibitem{tomiyama}
{\sc J. Tomiyama},
{\it Tensor products of commutative Banach algebras},
{\rm  T${\rm \hat{o}}$hoku Math. J. (2), 12 (1960), 147-154}.


\bibitem{varopoulos} 
\textsc{N.Th.Varopoulos}, 
\textit{On a problem of A. Beurling},  
\textrm{J. Funct. Anal.  2 (1968), 24-30}.


\bibitem{w}
{\sc J. Wichmann},
{\it Bounded approximate units and bounded approximate identities},
{\rm Proc. Amer. Math. Soc. 41 (1973), 547-550}.


\end{thebibliography}
\end{document}